\documentclass{amsart}

\usepackage{amssymb}
\usepackage[english]{babel}
\usepackage{amsmath}
\usepackage{amsthm}
\usepackage{amsfonts}
\usepackage{calligra}
\usepackage{mathrsfs}
\usepackage{mathtools}
\usepackage{tensor}
\usepackage{appendix}

\newtheorem{theorem}{Theorem}[section]
\newtheorem{cor}[theorem]{Corollary}
\newtheorem{lemma}[theorem]{Lemma}
\newtheorem{prop}[theorem]{Proposition}
\newtheorem{question}[theorem]{Question}
\newtheorem{conjecture}[theorem]{Conjecture}

\theoremstyle{definition}
\newtheorem{definition}[theorem]{Definition}

\theoremstyle{remark}
\newtheorem{remark}[theorem]{Remark}

\numberwithin{equation}{section}

\newcommand{\R}{\mathbb R}
\newcommand{\C}{\mathbb C}
\newcommand{\N}{\mathbb N}
\newcommand{\SSS}{\mathcal{S}}
\newcommand{\W}{W^{1,n}(\mathbb R^n, \mathbb R^n)}

\newcommand{\LL}{L^2(\mathbb C, \mathbb C)}
\newcommand{\HR}{{\mathcal H}^1(\mathbb R^n)}
\newcommand{\Hz}{{\mathcal H}^1_z(\Omega)}
\newcommand{\ds}{\displaystyle}
\newcommand{\BMO}{\operatorname{BMO}(\mathbb R^n)}
\newcommand{\VMO}{\operatorname{VMO}(\mathbb R^n)}
\newcommand{\CO}{C_c^\infty(\mathbb R^n)}

\newcommand{\Aint}{-\nobreak \hskip-11pt \nobreak\int}
\newcommand{\B}{\mathscr{B}}
\newcommand{\M}{\mathscr{M}}
\newcommand{\PP}{\mathscr{P}}
\newcommand{\Q}{\mathcal{Q}}
\newcommand{\LLL}{\mathscr{L}}
\newcommand{\defeq}{\mathrel{:\mkern-0.25mu=}}
\newcommand{\eqdef}{\mathrel{=\mkern-0.25mu:}}
\DeclareMathOperator{\dive}{div}
\DeclareMathOperator{\curl}{curl}
\DeclareMathOperator{\D}{D}
\DeclareMathOperator{\supp}{supp}

\begin{document}

\title[On the Hardy space theory of compensated compactness quantities]{On the Hardy space  theory of compensated compactness quantities}


\author{Sauli Lindberg}
\address{}
\curraddr{Department of Mathematics, Universidad Aut\'{o}noma de Madrid, E-28049 Madrid, Spain; ICMAT CSIC-UAM-UCM-UC3M, E-28049 Madrid, Spain}
\email{sauli.lindberg@uam.es}
\thanks{The author was supported by the ERC Starting grant no. 307179. He also wishes to thank }


\date{}

\dedicatory{}

\begin{abstract}
We make progress on a problem of R. Coifman, P.-L. Lions, Y. Meyer, and S. Semmes from 1993 by showing that the Jacobian operator $J$ does not map $W^{1,n}(\mathbb R^n,\mathbb R^n)$ onto the Hardy space $\mathcal{H}^1(\mathbb R^n)$ for any $n \ge 2$. The related question about surjectivity of $J \colon \dot{W}^{1,n}(\mathbb R^n,\mathbb R^n) \to \mathcal{H}^1(\mathbb R^n)$ is still open.

The second main result and its variants reduce the proof of $\mathcal{H}^1$ regularity of a large class of compensated compactness quantities to an integration by parts or easy arithmetic, and applications are presented. Furthermore, we exhibit a class of nonlinear partial differential operators in which weak sequential continuity is a strictly stronger condition than $\mathcal{H}^1$ regularity, shedding light on another problem of Coifman, Lions, Meyer, and Semmes.
\end{abstract}

\maketitle


\section{Introduction}
\textit{Compensated compactness} refers to the phenomenon of a nonlinear quantity enjoying nontrivial compactness or continuity properties due to, e.g., cancellations inherent in its algebraic structure. As an example relevant to this article, if $u^j \rightharpoonup u$ in $\W$, then the Jacobian determinants satisfy $J u^j \to J u$ in $\mathcal{D}'(\R^n)$ even though the individual terms of the determinants need not converge in $\mathcal{D}'(\R^n)$. Compensated compactness theory originated in the pioneering work of F. Murat and L. Tartar in the 1970's, and it has been used especially in the study of hyperbolic conservation laws. For information on compensated compactness see e.g. ~\cite{Dac82}, ~\cite{Daf10}, ~\cite{DiP85}, ~\cite{Eva90}, ~\cite{Hor97}, ~\cite{Lu03}, ~\cite{Mur78}, ~\cite{Mur79}, ~\cite{Tar79}, ~\cite{Tar83}, and the references contained therein.

The aim of this article is to advance the Hardy space theory of compensated compactness quantities which was initiated by R. Coifman, P.-L. Lions, Y. Meyer, and S. Semmes in the celebrated article ~\cite{CLMS93}. As shown in ~\cite{CLMS93}, numerous compensated compactness quantities such as Jacobians of mappings in $\W$ belong to the Hardy space $\HR$ (see \textsection \ref{H1 regularity of Jacobians}). One of the recurring themes of ~\cite{CLMS93} is the determination of the exact ranges of $\mathcal{H}^1$ regular compensated compactness quantities. In particular, the following problem is posed at ~\cite[p. 258]{CLMS93}.

\begin{question}[~\cite{CLMS93}] \label{CLMS question}
Does the Jacobian operator map $W^{1,2}(\R^2, \R^2)$ onto ${\mathcal H}^1(\R^2)$?
\end{question}

This article provides a negative answer to Question \ref{CLMS question} (see Theorem \ref{Main theorem} and Corollary \ref{Main corollary}). We also present two general non-surjectivity results, Theorems \ref{Bilinear theorem} and \ref{Extension to homogeneous polynomials}, that cover various other quantities. The surjectivity question is, however, still open for operators whose domain of definition is an \textit{homogeneous} Sobolev space, e.g. $J \colon \dot{W}^{1,n}(\R^n,\R^n) \to \HR$ (see \textsection \ref{Open questions}).

\vspace{0.2cm}
Another question posed in \cite{CLMS93} is whether compensated compactness and $\mathcal{H}^1$ regularity are, under natural conditions, equivalent -- see \textsection \ref{The equivalence between H1 regularity and weak sequential continuity} for the exact statement of the problem in ~\cite{CLMS93}. For a large and natural class of multilinear partial differential operators the answer is positive, as follows in an implicit way already from ~\cite[Theorem V.1]{CLMS93} and the work of R. Coifman and L. Grafakos in ~\cite{CG92} and ~\cite{Gra92}. We give a rather elementary proof of one version of this fact and formulate the result in Corollary \ref{Compensated integrability corollary}. In \textsection \ref{H1 regular quantities not enjoying weak continuity} we show that for more general homogeneous polynomials of partial derivatives, $\mathcal{H}^1$ regularity is a \textit{strictly weaker} condition than compensated compactness.

\vspace{0.2cm}

As a partial result concerning Question \ref{CLMS question}, Coifman, Lions, Meyer, and Semmes showed that $\HR$ is the minimal closed subspace of $\HR$ that contains $J u$ for every $u \in \W$. They achieved this by proving the following analogue of the classical atomic decomposition of $\HR$: every $h \in \HR$ can be expressed as $h = \sum_{j=1}^\infty \lambda^j J u^j$, where $\sum_{j=1}^\infty |\lambda^j| \lesssim \|h\|_{\mathcal{H}^1}$ and each $u^j \in \W$ satisfies $\|\D u^j\|_{L^n} \le 1$.\footnote{Here and elsewhere, the notation $X \lesssim Y$ means $|X| \le C Y$, where $C > 0$ is a constant that depends on parameters that are rather clear from context, in this instance the dimension $n$. In order to smoothen the exposition we will not display dependence on dimensions of Euclidean spaces or exponents $k$ and $p$ when presenting estimates in terms of Sobolev norms $\|\cdot\|_{W^{k,p}}$ or seminorms $\|\nabla^k \cdot\|_{L^p}$, but such dependence is always tacitly assumed. The notation $X \lesssim_B Y$ is used when we wish to emphasize the dependence of the implicit constant $C$ on a parameter $B$.} We show that the condition $\|\operatorname{D} u^j\|_{L^n} \le 1$ cannot be replaced by $\|u^j\|_{W^{1,n}} \le 1$.

\begin{theorem} \label{Main theorem}
For every $n \ge 2$, the set
\[\left\{ \sum_{j=1}^\infty \lambda^j J u^j \colon \|u^j\|_{W^{1,n}} \le 1 \text{ for every } j \in \N, \, \; \sum_{j=1}^\infty |\lambda^j| < \infty \right\}\]
is of the first category in $\HR$.
\end{theorem}

Since the Jacobian determinant is $n$-homogeneous, Theorem \ref{Main theorem} can be stated equivalently by saying that $\{\sum_{j=1}^\infty J u^j \colon \sum_{j=1}^\infty \|u^j\|_{W^{1,n}}^n < \infty\}$ is of the first category in $\HR$.

\begin{cor} \label{Main corollary}
The Jacobian operator does not map $\W$ onto $\HR$ for any $n \ge 2$.
\end{cor}

Most of the further examples of $\mathcal{H}^1$ regular compensated compactness quantities studied in ~\cite{CLMS93} are bilinear operators or null Lagrangians. The former class is treated in Theorem \ref{Bilinear theorem} and \S \ref{Proof of Theorem}--\ref{Applications to bilinear quantities} and the latter in \S \ref{Homogeneous polynomials of partial derivatives}. We next describe the setting of Theorem \ref{Bilinear theorem}.

When $k \ge 0$, $n \ge 2$, $m,N \ge 1$, $1 < p < \infty$, and $A \colon \mathcal{D}'(\R^n,\R^m) \to \mathcal{D}'(\R^n,\R^N)$ is a linear constant-coefficient, homogeneous partial differential operator (so that all the components of $A$ are of the form $\sum_{i=1}^m \sum_{|\alpha|=q} c_{\alpha,i} \partial^\alpha v_i$ with the same $q \in \N$), we denote
\begin{equation} \label{WA}
\begin{array}{lcl}
W_A^{k,p}(\R^n,\R^m) &\defeq& \{v \in W^{k,p}(\R^n,\R^m) \colon A v = 0\}, \\
C_{c,A}^\infty(\R^n,\R^m) &\defeq& \{v \in C_c^\infty(\R^n,\R^m) \colon A v = 0\}.
\end{array}
\end{equation}
Typical examples of $A$ in the Hardy space theory of compensated compactness are divergence, curl and the trivial operator $A = 0$.

Various compensated compactness quantities are bilinear partial differential operators $\B \colon W_{A_1}^{k_1,p_1}(\R^n,\R^{m_1}) \times W_{A_2}^{k_2,p_2}(\R^n,\R^{m_2}) \to L^1(\R^n)$ of the form
\begin{equation} \label{Form of a bilinear operator}
\B(u,v) \defeq \sum_{i=1}^{m_1} \sum_{j=1}^{m_2} \sum_{|\alpha| \le k_1, |\beta| \le k_2} c_{i,j,\alpha,\beta} \partial^\alpha u_i \partial^\beta v_j,
\end{equation}
where
\begin{equation} \label{Conditions on exponents}
k_1, k_2 \ge 0, \; n \ge 2, \; m_1, m_2, N \ge 1, \; p_1,p_2 \in ]1,\infty[ \text{ with  } \frac{1}{p_1} + \frac{1}{p_2} = 1
\end{equation}
(see \textsection \ref{H1 regularity of specific operators} for several examples).

\begin{theorem} \label{Bilinear theorem}
Suppose $\B \colon W^{k_1,p_1}_{A_1}(\R^n,\R^{m_1}) \times W_{A_2}^{k_2,p_2}(\R^n,\R^{m_2}) \to L^1(\R^n)$ is of the form \eqref{Form of a bilinear operator}, where the exponents satisfy \eqref{Conditions on exponents} and $A_1 \in \{0,\dive,\curl\}$. The following conditions are equivalent.
\renewcommand{\labelenumi}{(\roman{enumi})}
\begin{enumerate}

\item For every $(u,v) \in W_{A_1}^{k_1,p_1}(\R^n,\R^{m_1}) \times W_{A_2}^{k_2,p_2}(\R^n,\R^{m_2})$,
\[\int_{\R^n} \B(u,v) = 0.\]

\item For every $(u,v) \in W_{A_1}^{k_1,p_1}(\R^n,\R^{m_1}) \times W_{A_2}^{k_2,p_2}(\R^n,\R^{m_2})$,
\[\|\B(u,v)\|_{\mathcal{H}^1} \lesssim \|u\|_{W^{k_1,p_1}} \|v\|_{W^{k_2,p_2}}.\]
\end{enumerate}

\noindent Whenever $\B$ satisfies conditions $\operatorname{(i)--(ii)}$ and $k_1 \ge 1$, the set
\begin{equation} \label{First category set}
\left\{ \sum_{l=1}^\infty \lambda^l \B(u^l,v^l) \colon \sup_{l \in \N} \|u^l\|_{W^{k_1,p_1}} \|v^l\|_{W^{k_2,p_2}} \le 1, \; \sum_{l=1}^\infty |\lambda^l| < \infty \right\}
\end{equation}
is of the first category in $\HR$.
\end{theorem}

When $|\alpha|=k_1$ and $|\beta|=k_2$ in all the terms of $\B$, the estimate in condition (ii) can be strengthened to $\|\B(u,v)\|_{\mathcal{H}^1} \lesssim \|\nabla^{k_1} u\|_{L^{p_1}} \|\nabla^{k_2} v\|_{L^{p_2}}$ -- this fact is recorded in Theorem \ref{Homogeneous bilinear theorem}. When $A_1 = A_2 = 0$, we give in Proposition \ref{Coefficient proposition} another condition equivalent to (i) and (ii) that refers directly to the coefficients of $\B$ and is rather easy to test.

This article is organized as follows. In \textsection \ref{Background} we present the necessary background for the article, whereas \textsection \ref{Proof of Theorem A} is devoted to Theorem \ref{Main theorem}. In \textsection \ref{Proof of Theorem} we prove Theorem \ref{Bilinear theorem} and its modification for homogeneous Sobolev spaces, and applications are presented in \S \ref{Applications to bilinear quantities}. Section \ref{Homogeneous polynomials of partial derivatives} deals with null Lagrangians and more general homogeneous polynomials of partial derivatives, and open problems related to Question \ref{CLMS question} are discussed in \textsection \ref{Open questions}.

\section{Background} \label{Background}

\subsection{$\mathcal{H}^1$ regularity of Jacobians} \label{H1 regularity of Jacobians}
Good general references on the function spaces $\HR$ and $\BMO$ include  ~\cite{Gra04} and ~\cite{Ste93}. We fix, for the rest of this article, a function $\eta \in C_c^\infty(\R^n)$ with integral $\int_{\R^n} \eta \neq 0$ and denote $\eta_t(x) \defeq t^{-n} \eta(x/t)$ for all $x \in \R^n$ and $t > 0$. The real-variable \textit{Hardy space} is defined by
\[\HR \defeq \left\{ h \in {\mathcal S}'(\R^n) \colon \sup_{t>0} |h * \eta_t| \in L^1(\R^n) \right\}\]
and endowed with the norm $\|h\|_{\mathcal{H}^1} \defeq \|\sup_{t>0} |h * \eta_t|\|_{L^1}$. As shown by C. Fefferman and E. Stein in ~\cite{FS72}, any two choices of $\eta$ give equivalent norms for $\HR$. By a famous theorem of C. Fefferman, $\BMO$ is the dual of $\HR$. The space $\VMO$ can be defined as the closure of $\CO$ in $\BMO$, and it was shown by R. Coifman and G. Weiss in ~\cite{CW77} that $\HR = \operatorname{VMO}(\R^n)^*$.

S. M\"{u}ller showed in ~\cite{Mul89} that if $u \in W^{1,n}_{loc}(\R^n,\R^n)$ with $J u \ge 0$ a.e. in an open set $\Omega$, then $J u \log(2 + J u) \in L^1_{loc}(\Omega)$, and motivated chiefly by M\"{u}ller's theorem, R. Coifman, P.-L. Lions, Y. Meyer, and S. Semmes proved the following result.

\begin{theorem}[~\cite{CLMS93}] \label{H1 regularity theorem}
If $u = (u_1,\dots,u_n) \in \dot{W}^{1,n}(\R^n, \R^n)$, then $J u$ belongs to $\HR$ and the estimate $\|J u\|_{{\mathcal H}^1} \lesssim \prod_{j=1}^n \Vert \nabla u_j \Vert_{L^n}$ holds.
\end{theorem}

T. Iwaniec and J. Onninen subsequently weakened the hypotheses of Theorem \ref{H1 regularity theorem} to $u \in W^{1,n-1}_{loc}(\R^n,\R^n)$ and $|\text{cof}(\D u)| \in L^{n/(n-1)}(\R^n)$ (and strengthened the conclusion to $\|J u\|_{\mathcal{H}^1} \lesssim \|\text{cof}(\D u)\|_{L^{n/(n-1)}}^{n/
(n-1)}$) in ~\cite{IO02}, answering a question raised by S. M\"{u}ller, T. Qi and B.S. Yan in ~\cite{MQY94}. For a unified way of proving Theorem \ref{H1 regularity theorem} and $\mathcal{H}^1$ regularity results for numerous other quantities, involving e.g. fractional Laplacians, see ~\cite{1609.08547}. We also restate, for further reference, the Jacobian decomposition theorem mentioned in the introduction.

\begin{theorem}[~\cite{CLMS93}] \label{Jacobian decomposition theorem}
For every $h \in \HR$ there exist constans $\lambda^j \in \R$ and mappings $u^j \in \W$ with $\|\D u^j\|_{L^n} \le 1$ such that
\[h = \sum_{j=1}^\infty \lambda^j J u^j \qquad \text{and} \qquad \sum_{j=1}^\infty |\lambda^j| \lesssim \|h\|_{{\mathcal H}^1}.\]
\end{theorem}

It would be impossible to cite here all the applications of Theorem \ref{H1 regularity theorem}, its variants, and the $\mathcal{H}^1$ regularity of other compensated compactness quantities, but we mention some examples. The higher regularity of Jacobians was first observed by H. Wente in his study of surfaces of constant mean curvature (see ~\cite{Wen69}), and Theorem \ref{H1 regularity theorem} implies an improvement of an elliptic regularity result from ~\cite{Wen69} (see ~\cite[p. 283]{CLMS93} or ~\cite[Theorem 3.4.1]{Hel02}). Other areas of application include, among others, weakly harmonic maps between manifolds (~\cite{GM12}, \cite{Hel02}), geometric analysis in $\R^n$ (~\cite{AIM09}, ~\cite{IM01}), and mathematical fluid dynamics (~\cite{Lio96}) -- the books cited above also contain lots of further references.

\subsection{Sobolev spaces}
We recall some properties of Sobolev spaces used in this article. For an introduction to homogeneous Sobolev spaces we refer to ~\cite{KS91} and ~\cite{SSN98}. When $n,m,k \in \N$ and $1 < p < \infty$, we equip the \textit{homogeneous Sobolev space} $\dot{W}^{k,p}(\R^n,\R^m) \defeq \{u \in L^1_{loc}(\R^n,\R^m) \colon \nabla^k u \in L^p(\R^n,\R^{knm})\}$ with the seminorm
\[\|u\|_{\dot{W}^{k,p}} \defeq \|\nabla^k u\|_{L^p} \defeq \left( \sum_{i=1}^m \sum_{|\alpha|=k} \int_{\R^n} |\partial^\alpha u_i|^p \right)^\frac{1}{p},\]
and $\dot{W}^{k,p}(\R^n,\R^m)$ is complete in $\|\cdot\|_{\dot{W}^{k,p}}$ (see ~\cite[p. 188]{SSN98}). In order to define, as in \eqref{WA},
\[\dot{W}_A^{k,p}(\R^n,\R^m) \defeq \{v \in \dot{W}^{k,p}(\R^n,\R^m) \colon A v = 0\}\]
without ambiguities we \textit{do not} identify polynomials of degree smaller than $k$. Smooth, compactly supported functions are dense in $\dot{W}^{k,p}(\R^n,\R^m)$ (see ~\cite[pp. 194--195]{SSN98}). Furthermore, $\dot{W}^{k,p}(\R^n,\R^m) \subset W^{k,p}_{loc}(\R^n,\R^m)$ (see ~\cite[Theorem 4.5.8]{Hor90}). Finally, we write $u^j \rightharpoonup u$ in $\dot{W}^{k,p}(\R^n,\R^m)$ when $\nabla^k u^j \rightharpoonup \nabla^k u$ in $L^p(\R^n,\R^{kn})$.

We will repeatedly use a higher-order Poincar\'{e} inequality which is presented in the following lemma. A proof can be found e.g. at ~\cite[pp. 2960-2961]{PT15}.

\begin{lemma} \label{Higher order Poincare lemma}
Let $Q \subset \R^n$ be a cube and suppose $f \in W^{k,p}(Q)$, where $k \in \N$ and $1 \le p < n$. Then there exists a polynomial $P^{k-1}_Q f$ of order at most $k-1$ such that
\[\Aint_Q \partial^\alpha P^{k-1}_Q f = \Aint_Q \partial^\alpha f\]
for every multi-index $\alpha \in (\N \cup \{0\})^n$ with $|\alpha| \le k-1$. Furthermore,
\begin{equation} \label{Poincare 1}
\left( \Aint_Q |f-P^{k-1}_Q f|^\frac{np}{n-kp} \right)^\frac{n-kp}{np} \lesssim \ds l(Q)^k \left( \Aint_Q |\nabla^k f|^p \right)^\frac{1}{p} \quad \text{ if } kp < n,
\end{equation}
\begin{equation} \label{Poincare 2}
\left( \Aint_Q |f-P^{k-1}_Q f|^q \right)^\frac{1}{q} \lesssim \ds l(Q)^k \left( \Aint_Q |\nabla^k f|^p \right)^\frac{1}{p} \quad \text{ if } kp \ge n \text{ and } q < \infty,
\end{equation}
where the implicit constants do not depend on $Q$.
\end{lemma}

In ~\cite{PT15}, M. Prats and X. Tolsa state a modified version of inequalities \eqref{Poincare 1} and \eqref{Poincare 2}, namely $\|f-P^{k-1}_Q f\|_{L^p(Q)} \lesssim l(Q)^k \|\nabla^k f\|_{L^p(Q)}$, but their proof also gives \eqref{Poincare 1} and \eqref{Poincare 2}.

\subsection{Complex analytical tools}
We recall some tools of planar harmonic analysis and refer to ~\cite{AIM09} for the details, see also ~\cite{Lin15}. The homogeneous Sobolev space $\dot{W}^{1,2}(\C,\C)$ is, up to identification of constant functions to zero, isomorphic to $\LL$, and an isomorphism is given by either of the \textit{Wirtinger derivatives}
\[u_z \defeq \frac{1}{2} \left( \frac{\partial}{\partial x} - i \frac{\partial}{\partial_y} \right) (u_1+iu_2) \quad \text{and} \quad u_{\bar{z}} \defeq \frac{1}{2} \left( \frac{\partial}{\partial x} + i \frac{\partial}{\partial_y} \right) (u_1+iu_2).\]
The \textit{Beurling transform} $\SSS \colon \LL \to \LL$ is given by the Cauchy principal value integral
\begin{equation} \label{Formula of the Beurling transform}
\SSS f(z) \defeq -\frac{1}{\pi} \lim_{\epsilon \searrow 0} \int_{\C \setminus B(0,\epsilon)} \frac{f(w)}{(z-w)^2} \, dw
\end{equation}
which converges a.e. when $f \in L^2(\C,\C)$. The Beurling transform is an isometry in $L^2(\C,\C)$, and $\SSS(u_{\bar{z}}) = u_z$ for all $u \in \dot{W}^{1,2}(\C,\C)$. Thus, when $u \in \dot{W}^{1,2}(\C,\C)$ and $f \defeq u_{\bar{z}} \in \LL$, we may write
\begin{equation} \label{Alternative representations}
J u = |u_z|^2 - |u_{\bar{z}}|^2 = |\SSS f|^2 - |f|^2.
\end{equation}

When $b \in \text{BMO}(\C)$, the complex linear operator $K_b \colon \LL \to \LL$,
\[K_b f \defeq \overline{(\SSS b - b \SSS) \overline{\SSS f}},\]
defined and studied in ~\cite{Lin15}, is self-adjoint. It allows us to rewrite a formula from ~\cite[p. 547]{AIM09} in the form
\begin{equation}  \label{AIM formula}
\int_\C b (|\SSS f|^2 - |f|^2) = \int_\C f \overline{K_b f}.
\end{equation}
By combining \eqref{Alternative representations} and \eqref{AIM formula}, the dual pairing of $b$ and $J u$ is expressed as an $L^2$ inner product. Furthermore, since $K_b$ is self-adjoint, its operator norm and numerical radius coincide, which implies the following characterization:
\begin{equation} \label{Definition of BMOS norm}
\|b\|_{\text{BMO}_\SSS} \defeq \sup_{\int_{\C} |u_{\bar{z}}|^2 = 1} \int_{\C} b J u = \sup_{\|f\|_{L^2}=1} \int_\C f \overline{K_b f} = \|K_b\|_{L^2 \to L^2}.
\end{equation}
The norm $\|\cdot\|_{\text{BMO}_\SSS}$ appeared, in a different guise, already in ~\cite{CLMS93} where it was shown to be equivalent to $\|\cdot\|_{\text{BMO}}$.

When $1 < p < \infty$, the results of this subsection have natural analogues for $L^p(\C,\C)$. In particular, \eqref{AIM formula} holds for every $b \in L^{p'}(\C)$ and $f \in L^{2p}(\C,\C)$.

\subsection{Null Lagrangians} \label{Null Lagrangians}
Jacobian, Hessian and minors of the Jacobian matrix belong to the class of null Lagrangians. In this article we confine ourselves to the null Lagrangians most relevant for compensated compactness theory, those of the form $\LLL(\nabla^k u)$. For further information on null Lagrangians we refer to ~\cite{BCO81}, \cite{Iwa07}, ~\cite{OS88}, and ~\cite{Olv93}.

Here and throughout the article we have striven, whenever possible, to use upper indices to enumerate elements of finite and infinite sequences and lower indices to enumerate components of multi-indices and mappings. We follow ~\cite{BCO81} but try to keep the notation consistent with the other parts of this article. In the rest of this subsection, $n,m,k \in \N$.

\begin{definition}
We denote by $X(n,m,k)$ the $m \times {n+k-1 \choose k}$-dimensional space of real matrices $V = [V_{i, \alpha}]$, where $1 \le i \le m$ and $\alpha \in (\N \cup \{0\})^n$ with $|\alpha|=k$. 
\end{definition}

The ordering of the components of elements of $X(n,m,k)$ is irrelevant for the arguments used in this article. When $u \in W^{k,r}(\R^n,\R^m)$, where $r \in ]1,\infty[$, we denote $\nabla^k u \defeq [\partial^\alpha u_i] \colon \R^n \to X(n,m,k)$.

\begin{definition}
A continuous function $\LLL \colon X(n,m,k) \to \R$ is called a \textit{null Lagrangian} if
\[\int_\Omega \LLL(\nabla^k(u+\varphi)) = \int_\Omega \LLL(\nabla^k u)\]
for every bounded open set $\Omega \subset \R^n$ and every $u \in C^k(\bar{\Omega})$ and $\varphi \in C_0^\infty(\Omega)$.
\end{definition}

Null Lagrangians are also called \textit{quasiaffine functions} in the literature. The following result contains part of ~\cite[Theorem 3.4]{BCO81}; condition (iii) is presented in a slightly different form in ~\cite{BCO81}.

\begin{theorem}[~\cite{BCO81}] \label{Ball-Currie-Olver theorem}
Let $\LLL \colon X(n,m,k) \to \R$ be continuous. The following conditions are equivalent:

\renewcommand{\labelenumi}{(\roman{enumi})}
\begin{enumerate}
\item $\LLL$ is a null Lagrangian.

\item $\int_\Omega \LLL(c+\nabla^k \varphi) = \int_\Omega \LLL(c)$ for every bounded open set $\Omega \subset \R^n$ and all $\varphi \in C_c^\infty(\Omega,\R^m)$ and $c \in X(n,m,k)$.

\item $\LLL$ is a polynomial (of degree $r$, say) and furthermore, whenever $u^j \rightharpoonup u$ in $\dot{W}^{k,r}(\R^n,\R^m)$, we have $\LLL(\nabla^k u^j) \to \LLL(\nabla^k u)$ in $\mathcal{D}'(\R^n)$.
\end{enumerate}
\end{theorem}

Another theorem from ~\cite{BCO81} says that $\LLL$ is a null Lagrangian if and only if $\LLL$ is an affine combination of minors of $\nabla^k u$. This characterization implies the following result which is also a special case of Theorem \ref{Extension to homogeneous polynomials}.

\begin{prop} \label{H1 regularity of null Lagrangians}
Suppose that a null Lagrangian $\LLL \colon X(n,m,k) \to \R$ is a polynomial of degree $r \ge 2$. The inequality $\|\LLL(\nabla^k u)\|_{\mathcal{H}^1} \lesssim \|\nabla^k u\|_{L^r}^r$ holds for all $u \in \dot{W}^{k,r}(\R^n,\R^m)$ if and only if $\LLL$ is $r$-homogeneous.
\end{prop}

In \textsection \ref{Homogeneous polynomials of partial derivatives} we give an example of an homogeneous polynomial $\LLL \colon X(3,2,2) \to \R$ that satisfies the estimate $\|\LLL(\nabla^2 u)\|_{\mathcal{H}^1} \lesssim \|\nabla^2 u\|_{L^3}^3$ for all $u \in \dot{W}^{2,3}(\R^3,\R^2)$ but fails to be be a null Lagrangian.

\subsection{Potentials for divergence and curl} \label{Potential for divergence and curl}

It is the purpose of this subsection to find, given $k \in \N \cup \{0\}$ and $1 < p < \infty$, a first-order homogeneous partial differential operator $B \colon \dot{W}^{k+1,p}(\R^n,\R^{n(n-1)/2}) \to \dot{W}^{k,p}_{\dive}(\R^n,\R^n)$ such that every $u \in \dot{W}^{k,p}_{\dive}(\R^n,\R^n)$ can be written as $u = B \Phi$, where $\|\nabla^{k+1} \Phi\|_{L^p} \approx \|\nabla^k u\|_{L^p}$. The result is in effect a Poincar\'{e} lemma with suitable norm control: given $u \in \dot{W}^{k,p}_{\dive}(\R^n,\R^n)$ the idea is to define the closed $n-1$ form $\alpha \defeq * (\sum_{j=1}^n u_j \, dx_j)$ and then find a suitable $n-2$ form $\beta$ with $d \beta = \alpha$. We will however avoid the use of differential forms in the exposition. We also present analogous results for curl.

\vspace{0.2cm}
We denote the vector space of strictly upper triangular $n \times n$ matrices by $\R^{n \times n}_\Delta$.

\begin{definition}
We define $C \colon \mathcal{D}'(\R^n,\R^n) \to \mathcal{D}'(\R^n,\R^{n \times n}_\Delta)$ by
\[C(T) \defeq [(-1)^{i+j} \chi_{i<j}(i,j) \, (\partial_j T_i - \partial_i T_j)]_{i,j=1}^n\]
and $B \colon \mathcal{D}'(\R^n,\R^{n \times n}_\Delta) \to \mathcal{D}'_{\dive}(\R^n,\R^n)$ by
\[B(\Phi)
\defeq \prod_{j=1}^n \left( \sum_{i < j} (-1)^{i+j} \partial_i \Phi_{ij} + \sum_{i>j} (-1)^{i+j+1} \partial_i \Phi_{ji} \right).\]
\end{definition}
A straightforward computation gives the identity $-\Delta T = B \circ C(T)$ for all divergence-free $T \in \mathcal{D}'(\R^n,\R^n)$. In order to use the identity we prove the following lemma.

\begin{lemma} \label{Poisson equation lemma}
Let $k \in \N \cup \{0\}$ and $1 < q < n$. If $u \in L^q_{\dive}(\R^n,\R^n)$, then the Poisson equation $-\Delta T = u$ has a solution $T \in \dot{W}^{2,q}_{\dive}(\R^n,\R^n) \cap \dot{W}^{1,nq/(n-q)}(\R^n,\R^n)$ that is unique up to an additive constant.
\end{lemma}

\begin{proof}
We first prove the existence statement. Let $u \in L^q_{\dive}(\R^n,\R^n)$ and choose $\varphi^k \in C_{c,\dive}^\infty(\R^n,\R^n)$ such that $\|\varphi^k - u\|_{L^q} \to 0$ as $k \to \infty$ (see e.g. ~\cite[Lemma III.2.1]{Gal11}). Now the sequence $(-\Delta[\varphi^k * G])_{k=1}^\infty = (\varphi^k)_{k=1}^\infty$ is Cauchy in $L^q(\R^n,\R^n)$, where $G = -(2 \pi)^{-1} \log |\cdot|$ when $n = 2$ and $G = (n(n-2) \omega_n)^{-1} |\cdot|^{2-n}$ when $n \ge 3$. Thus $(\nabla^2 [\varphi^k * G])_{k=1}^\infty$ is Cauchy in $L^q(\R^n,\R^{n^3})$.

Now $\D[\varphi^k * G] \in L^{nq/(n-q)}(\R^n,\R^{n \times n})$ for every $k \in \N$. By the Sobolev-Gagliardo-Nirenberg theorem, $(\D[\varphi^k * G])_{k=1}^\infty$ is Cauchy in $L^{nq/(n-q)}(\R^n,\R^{n \times n})$. Hence, for every $i \in \{1,\dots,n\}$ there exists $f_i \in \dot{W}^{1,q}(\R^n,\R^n) \cap L^{nq/(n-q)}(\R^n,\R^n)$ such that $\|\nabla^2 [\varphi^k_i * G] - \D f_i\|_{L^q} \to 0$ and $\|\nabla [\varphi^k_i * G] - f_i\|_{L^{nq/(n-q)}} \to 0$. Now $f_i$ is curl-free and thus $f_i = \nabla T_i$, where $T_i \in  \dot{W}^{2,q}(\R^n) \cap \dot{W}^{1,nq/(n-q)}(\R^n)$. Consequently, $T \defeq (T_1,\dots,T_n)$ satisfies $\dive T = \lim_{k \to \infty} \dive \varphi^k * G = 0$ and $-\Delta T = u$.

The uniqueness statement is proved as follows. Suppose $\tilde{T} \in \dot{W}^{2,q}_{\dive}(\R^n,\R^n) \cap \dot{W}^{1,nq/(n-q)}(\R^n,\R^n)$ satisfies $\Delta \tilde{T} = 0$. Then $\D \tilde{T} \in L^{nq/(n-q)}(\R^n,\R^{n \times n})$ is harmonic and therefore vanishes, and so $\tilde{T}$ is constant.
\end{proof}

Before producing the potential for divergence recall that when $1 < p < \infty$ and $R = (R_1,\dots,R_n)$, the operator $I+R \otimes R \colon L^p(\R^n,\R^n) \to L^p(\R^n,\R^n)$ is a projection onto divergence-free vector fields.

\begin{prop} \label{Potential proposition}
Let $k \in \N \cup \{0\}$ and $1 < p < \infty$. The linear operator $B \colon \dot{W}^{k+1,p}(\R^n,\R^{n \times n}_\Delta) \to \dot{W}^{k,p}_{\dive}(\R^n,\R^n)$ has a linear right inverse $u \mapsto \Phi$ such that $\|\nabla^{k+1} \Phi\|_{L^p} \approx \|\nabla^k u\|_{L^p}$ for all $u \in \dot{W}^{k,p}_{\dive}(\R^n,\R^n)$. Furthermore, $\nabla^k u^j \rightharpoonup \nabla^k u$ in $L^p(\R^n,\R^{kn})$ if and only if $\nabla^{k+1} \Phi^j \rightharpoonup \nabla^{k+1} \Phi$ in $L^p(\R^n,\R^{(k+1)n(n-1)/2})$.
\end{prop}

\begin{proof}
It suffices to show that the linear right inverse exists in a dense subset of $\dot{W}^{k,p}_{\dive}(\R^n,\R^n)$. Seeking a suitable dense set, let $u \in \dot{W}^{k,p}_{\dive}(\R^n,\R^n)$ and choose $u^j \in C_c^\infty(\R^n,\R^n)$ such that $\|\nabla^k(u^j-u)\|_{L^p} \to 0$. When $\alpha \in (\N \cup \{0\})^n$ with $|\alpha|=k$, note that $\partial^\alpha u \defeq (\partial^\alpha u_1,\dots, \partial^\alpha u_n) \in L^p_{\dive}(\R^n,\R^n)$ and thus $(I+ R \otimes R) \partial^\alpha u = \partial^\alpha u$. Now
\[\partial^\alpha (I + R \otimes R) u^j = (I + R \otimes R) \partial^\alpha u^j \to (I + R \otimes R) \partial^\alpha u = \partial^\alpha u\]
in $L^p(\R^n,\R^n)$. We conclude that $\{(I + R \otimes R) \varphi \colon \varphi \in C_c^\infty(\R^n,\R^n)\}$ is dense in $\dot{W}^{k,p}_{\dive}(\R^n,\R^n)$.

Fix now $q \in ]1,n[$ and let $\varphi \in C_c^\infty(\R^n,\R^n)$. Use Lemma \ref{Poisson equation lemma} to choose a solution $T \in \dot{W}^{2,q}_{\dive}(\R^n,\R^n) \cap \dot{W}^{§1,nq/(n-q)}(\R^n,\R^n)$ of the Poisson equation $-\Delta T = (I + R \otimes R) \varphi$. Now $\Phi \defeq	C(T)$ is well-defined and $B \Phi = (I + R \otimes R) \varphi$. Moreover, $(I + R \otimes R) \varphi \mapsto \Phi$ is linear. Trivially $\|\nabla^{k+1} \Phi\|_{L^p} \lesssim \|\nabla^{k+2} T\|_{L^p}$, whereas for every $\alpha \in (\N \cup \{0\})^n$ with $|\alpha|=k+2$ we write $\alpha = (\alpha-\gamma)+\gamma$ with $|\gamma|=2$ and obtain
\[\|\partial^\alpha T\|_{L^p}
= \|R^\gamma \partial^{\alpha-\gamma} (I + R \otimes R) \varphi\|_{L^p}
\lesssim \|\nabla^k (I + R \otimes R) \varphi\|_{L^p},\]
where $R^\gamma$ is the second-order Riesz transform corresponding to $\gamma$. The converse estimate $\|\nabla^k (I + R \otimes R) \varphi\|_{L^p} = \|\nabla^k B \Phi\|_{L^p} \lesssim \|\nabla^{k+1} \Phi\|_{L^p}$ is obvious.

\vspace{0.2cm}
We next show the equivalence of weak convergences. Clearly $\nabla^{k+1} \Phi^j \rightharpoonup \nabla^{k+1} \Phi$ implies $\nabla^k u^j \rightharpoonup \nabla^k u$. For the other direction suppose $\nabla^k u^j \rightharpoonup \nabla^k u$ and let $\epsilon > 0$. Choose $u_\epsilon \in C_{c,\dive}^\infty(\R^n,\R^n)$ such that $\|\nabla^k u_\epsilon - \nabla^k u\|_{L^p} < \epsilon$. For every $j \in \N$ select $u^j_\epsilon \approx u_\epsilon + u^j - u$ in $C_{c,\dive}^\infty(\R^n,\R^n)$ such that $\|\nabla^k u^j_\epsilon - \nabla^k u^j\|_{L^p} < 2 \epsilon$ and $\nabla^k u^j_\epsilon \rightharpoonup \nabla^k u_\epsilon$. It now suffices to note that $\nabla^k u^j_\epsilon \rightharpoonup \nabla^k u_\epsilon$ implies, as above, $\partial^\alpha T^j_\epsilon = R^\gamma \partial^{\alpha-\gamma} u^j_\epsilon \rightharpoonup R^\gamma \partial^{\alpha-\gamma} u_\epsilon = \partial^\alpha T_\epsilon$ for every $\alpha \in (\N \cup \{0\})^n$ with $|\alpha|=k+2$, which in turn implies that $\nabla^{k+1} \Phi^j_\epsilon \rightharpoonup \nabla^{k+1} \Phi_\epsilon$. We conclude that $\nabla^{k+1} \Phi^j \rightharpoonup \nabla^{k+1} \Phi$.
\end{proof}

In \textsection \ref{The case A1 in div, curl} we need the following variant of Proposition \ref{Potential proposition}.

\begin{prop} \label{Potential proposition 2}
Let $k \in \N \cup \{0\}$ and $1 < p < \infty$. The linear operator $B \colon \cap_{l=1}^{k+1}\dot{W}^{l,p}(\R^n,\R^{n \times n}_\Delta) \to W^{k,p}_{\dive}(\R^n,\R^n)$ has a linear right inverse $u \mapsto \Phi$ such that $\sum_{l=1}^{k+1} \|\nabla^l \Phi\|_{L^p} \approx \|u\|_{W^{k,p}}$ for all $u \in W^{k,p}_{\dive}(\R^n,\R^n)$.
\end{prop}

\begin{proof}
When $u \in W^{k,p}_{\dive}(\R^n,\R^n)$, again we choose mappings $u^j \in C_c^\infty(\R^n,\R^n)$ such that $\|u^j-u\|_{W^{k,p}} \to 0$, solve the Poisson equation $-\Delta T^j = (I + R \otimes R) u^j$ and set $\Phi^j \defeq C (T^j)$. Now $(\Phi^j)_{j=1}^\infty$ is a Cauchy sequence in $\dot{W}^{l,p}(\R^n,\R^{n \times n}_\Delta)$ for every $l \in \{1,\dots,k+1\}$ and therefore $(\D \Phi^j)_{j=1}^\infty$ has a limit $w$ in $W^{k,p}(\R^n,\R^{n^2(n-1)/2})$. On the other hand, there exists $\Phi \in \dot{W}^{1,p}(\R^n,\R^{n \times n}_\Delta)$ such that $\|\D \Phi^j - \D \Phi\|_{L^p} \to 0$ and so $\D \Phi = w \in W^{k,p}(\R^n,\R^{n^2(n-1)/2})$. As in the proof above, $\sum_{l=1}^{k+1} \|\nabla^l \Phi\|_{L^p} \approx \|\D \Phi\|_{W^{k,p}} \approx \|u\|_{W^{k,p}}$.
\end{proof}

Analogous results of course hold for $ W^{k,p}_{\curl}(\R^n,\R^n)$ and $\dot{W}^{k,p}_{\curl}(\R^n,\R^n)$.

\begin{prop} \label{Potential proposition 3}
Let $k \in \N \cup \{0\}$ and $1 < p < \infty$. If $u \in \dot{W}^{k,p}_{\curl}(\R^n,\R^n)$, then there exists $\Psi \in \dot{W}^{k+1,p}(\R^n)$ such that $\nabla \Psi = u$ and $\|\nabla^{k+1} \Psi\|_{L^p} \approx \|\nabla^k u\|_{L^p}$. If, furthermore, $u \in W^{k,p}_{\curl}(\R^n,\R^n)$, then $\sum_{l=1}^{k+1} \|\nabla^l \Psi\|_{L^p} \approx \|u\|_{W^{k,p}}$.
\end{prop}

We also record the following fact.

\begin{prop} \label{Density proposition}
Suppose $k \in \N \cup \{0\}$ and $1 < p < \infty$. Then $C_{c,\dive}^\infty(\R^n,\R^n)$ is dense in $\dot{W}^{k,p}_{\dive}(\R^n,\R^n)$ and $C_{c,\curl}^\infty(\R^n,\R^n)$ is dense in $\dot{W}^{k,p}_{\curl}(\R^n,\R^n)$.
\end{prop}

\begin{proof}
Let $u \in \dot{W}^{k,p}_{\dive}(\R^n,\R^n)$ and choose $\Phi \in \dot{W}^{k+1,p}(\R^n,\R^{n \times n}_\Delta)$ with $B \Phi = u$. Select $\Phi^j \in C_c^\infty(\R^n,\R^{n \times n}_\Delta)$ such that $\Phi^j \to \Phi$ in $\dot{W}^{k+1,p}(\R^n,\R^{n \times n}_\Delta)$. Then $C_{c,\dive}^\infty(\R^n,\R^n) \ni B(\Phi^j) \to B(\Phi) = u$ in $\dot{W}^{k,p}_{\dive}(\R^n,\R^n)$. A similar proof applies to $\dot{W}^{k,p}_{\curl}(\R^n,\R^n)$.
\end{proof}

\section{Proof of Theorem \ref{Main theorem}} \label{Proof of Theorem A}
Theorem \ref{Main theorem} is proved in this section, and we divide the proof into two parts. In \S \ref{A functional analytical lemma} we present an extension of a functional analytical lemma from ~\cite{CLMS93}, and the proof is finished in \S \ref{Proof of the main result}. In \S \ref{Failure of norm control} we show that, as a corollary of Theorem \ref{Main theorem}, we have very little control on the norms of solutions of the Jacobian equation.

\subsection{A functional analytical lemma} \label{A functional analytical lemma}
In the proof of Theorem \ref{Jacobian decomposition theorem} Coifman, Lions, Meyer, and Semmes used functional analytical lemmas ~\cite[Lemmas III.1-III.2]{CLMS93} which give equivalent conditions for the existence of atomic decompositions in Banach spaces. We add to their criteria a further characterization that can be used when an atomic decomposition does \textit{not} exist.

When $X$ is a Banach space and $V \subset X$ is bounded, recall that the \textit{s-convex hull} of $V$ is defined by
\[\operatorname{s}(V) \defeq \left\{ \sum_{j=1}^\infty \lambda^j x^j \colon \lambda^j \ge 0 \text{ and } x^j \in V \text{ for all } j \in \N, \; \sum_{j=1}^\infty \lambda^j = 1 \right\}.\]
It is easy to show that $\text{co}(V) \subset \operatorname{s}(V) \subset \overline{\text{co}}(V)$, and thus $\overline{\operatorname{s}(V)} = \overline{\operatorname{co}}(V)$.

\begin{lemma} \label{Main lemma}
Suppose $V$ is a symmetric, bounded subset of a Banach space $X$. The following conditions are equivalent.
\begin{enumerate}
\renewcommand{\labelenumi}{(\roman{enumi})}
\item $\|l\|_{X^*}$ and $\sup_{x \in V} \langle l, x \rangle$ are equivalent norms in $X^*$.
\item $\overline{\operatorname{co}}(V)$ contains a ball centered at the origin.
\item $\operatorname{s}(V)$ contains a ball centered at the origin.
\item The set $\{\sum_{j=1}^\infty \lambda^j x^j \colon x^j \in V \; \forall j \in \N, \; \sum_{j=1}^\infty |\lambda^j| < \infty\}$ is of the second category in $X$.
\end{enumerate}
\end{lemma}

The equivalence of $(i)-(iii)$ is proved at ~\cite[pp. 262-264]{CLMS93}, and direction $(iii) \Rightarrow (iv)$ follows from the Baire category theorem. In this article we only use implication $(iv) \Rightarrow (i)$.

\begin{proof}[Proof of direction $(iv) \Rightarrow (i)$]
Suppose (iv) holds. Since $V$ is symmetric,
\[\left\{ \sum_{j=1}^\infty \lambda^j x^j \colon x^j \in V \; \forall j \in \N, \; \sum_{j=1}^\infty |\lambda^j| < \infty \right\} = \bigcup_{k=1}^\infty k \operatorname{s}(V).\]
By assumption, there exists $k \in \N$ such that $\overline{k \operatorname{s}(V)} = k \overline{\operatorname{co}}(V)$ contains a ball $B(x_0,\delta)$, and so $B(0,2\delta) = B(x_0,\delta) - B(x_0,\delta) \subset 2 k \overline{\operatorname{co}}(V)$. When $l \in X^*$, we use the previous set inclusion and the boundedness of $V$ to get, as in ~\cite{CLMS93},
\[\frac{\delta}{k} \|l\|_{X^*} \le \sup_{x \in \overline{\operatorname{co}}(V)} \langle l,x \rangle = \sup_{x \in \overline{\operatorname{co}}(V)} |\langle l,x \rangle| = \sup_{x \in V} |\langle l,x \rangle| = \sup_{x \in V} \langle l,x \rangle \le C \|l\|_{X^*}.\]
\end{proof}

\subsection{Completion of the proof} \label{Proof of the main result}
We denote
\[V \defeq \left\{ J u \colon \|u\|_{W^{1,n}} \le 1 \right\} \subset \HR\]
and set out to show that the norms $\|b\|_{\text{BMO}}$ and $\sup_{h \in V} \int_{\R^n} bh$ are not equivalent in $\BMO$. Theorem \ref{Main theorem} then follows from Lemma \ref{Main lemma}. The basic phenomenon behind the proof is the incompatibility of the scaling properties of $\W$ and $\BMO = \HR^*$.

When $\tau > 0$, the change of variables $y = \tau^{-1} x$ yields the formulas
\begin{equation} \label{tau formula for BMO}
\|b(\tau \cdot)\|_{\text{BMO}} = \|b\|_\text{BMO},
\end{equation}
\begin{equation} \label{tau formula for Sobolev}
\int_{\R^n} |u(\tau \cdot)|^n + \sum_{i,j=1}^n \int_{\R^n} |\partial_j [u_i(\tau \cdot)]|^n = \tau^{-n} \int_{\R^n} |u|^n + \sum_{i,j=1}^n \int_{\R^n} |\partial_j u_i|^n,
\end{equation}
\begin{equation} \label{tau formula for change of varables}
\int_{\R^n} b(\tau \cdot) J [u(\tau \cdot)] = \int_{\R^n} b J u
\end{equation}
for all $b \in \BMO$ and $u \in \W$.

\begin{proof}[Proof of Theorem \ref{Main theorem}]
Seeking contradiction, assume that
\begin{equation} \label{Counter-assumption}
\|b\|_\text{BMO} \lesssim \sup_{\|u\|_{W^{1,n}} \le 1} \int_\C b J u = \sup_{u \neq 0} \frac{\int_\C b J u}{\|u\|_{W^{1,n}}^n}
\end{equation}
for all $b \in \BMO$. Fix $b \in \CO$ with $\|b\|_{\text{BMO}} = 1$.

Let $\tau > 0$. By \eqref{tau formula for BMO}-\eqref{Counter-assumption}, there exists $v = u(\tau \cdot) \in \W \setminus \{0\}$ such that
\begin{equation} \label{False for small tau}
\begin{array}{lcl}
    1
&=& \ds \|b(\tau \cdot)\|_{\text{BMO}}
\lesssim \frac{\int_\C b(\tau \cdot) J v}{\int_{\R^n} |v|^n + \sum_{i,j=1}^n \int_{\R^n} |\partial_j v_i|^n} \\
&=& \ds \frac{\int_{\R^n} b J u}{\tau^{-n} \int_{\R^n} |u|^n + \sum_{i,j=1}^n \int_{\R^n} |\partial_j u_i|^n}.
  \end{array}
\end{equation}
In \eqref{False for small tau} the mapping $u$ depends on $\tau$. Our goal is to show that when $\tau > 0$ is small enough, \eqref{False for small tau} leads to a contradiction.

Let $M > 0$ and use integration by parts and Young's inequality to get
\[\begin{array}{lcl}
      \displaystyle \int_{\R^n} b J u
&=&   \displaystyle - \int_{\R^n} u_1 J(b,u_2,\dots,u_n) \\
&\le& \displaystyle \int_{\R^n} \left( \frac{1}{n} (M |u_1|)^n + \frac{n-1}{n} \left( \frac{|J(b,u_2,\dots,u_n)|}{M} \right)^\frac{n}{n-1} \right) \\
&\le& \displaystyle \frac{1}{n} M^n \int_{\R^n} |u|^n + \frac{n-1}{n} \frac{\|\nabla b\|_{L^\infty}^\frac{n}{n-1}}{M^\frac{n}{n-1}} \int_{\R^n} |\D u|^n,
  \end{array}\]
where we used the Hadamard inequality $|J(b,u_2,\dots,u_n)| \le |\nabla b| \prod_{j=2}^n |\nabla u_j|$. By combining the previous estimate with \eqref{False for small tau} we get
\begin{equation} \label{False inequality}
\begin{array}{lcl}
& &  \ds \tau^{-n} \int_{\R^n} |u|^n + \sum_{i,j=1}^n \int_{\R^n} |\partial_j u_i|^n \\
&\le& \ds C(n) \left( M^n \int_{\R^n} |u|^n + \frac{\|\nabla b\|_{L^\infty}^\frac{n}{n-1}}{M^\frac{n}{n-1}} \sum_{i,j=1}^n \int_{\R^n} |\partial_j u_i|^n \right).
  \end{array}
\end{equation}
When we choose $M$ large enough and then $\tau$ small enough (both depending on $b$), a mapping $u \in \W \setminus \{0\}$ satisfying \eqref{False inequality} cannot exist, and so we have obtained a contradiction.
\end{proof}

\subsection{Failure of norm control} \label{Failure of norm control}
Despite Theorem \ref{Main theorem} and Corollary \ref{Main corollary}, the range of $J \colon \W \to \HR$ is dense in $\HR$. This follows from the remarkable work of G. Cupini, B. Dacorogna and O. Kneuss on the Jacobian equation in ~\cite[Theorem 1]{CDK09} and more directly from ~\cite[Theorem 1]{Kne12}. However, solutions of the Jacobian equation utterly lack norm control, as we show in this subsection.

A mapping $u \in \W$ is said to be a \textit{minimum norm solution} if $\|u\|_{W^{1,n}} = \min \{\|v\|_{W^{1,n}} \colon J v = J u\}$. If $h \in \HR$ and the Jacobian equation $J u = h$ has a solution, then the equation has a minimum norm solution; this follows from the weak continuity of the Jacobian and the direct method of the calculus of variations.

\begin{prop} \label{Proposition on failure of norm control}
Let $C > 0$. Then the set
\[\{J u \colon u \text{ is a minimum norm solution, } \|u\|_{W^{1,n}}^n \le C \|J u\|_{\mathcal{H}^1}\}\]
is nowhere dense in $\HR$ and the set
\[\{J u \colon u \text{ is a minimum norm solution, } \|u\|_{W^{1,n}}^n > C \|J u\|_{\mathcal{H}^1}\}\]
is dense in $\HR$.
\end{prop}

\begin{proof}
We use a proof by contradiction to deduce the first statement. Suppose $\overline{\{J u \colon \|u\|_{W^{1,n}}^n \le C \|J u\|_{\mathcal{H}^1}}\}$ contains a ball. Then, for some $M > 0$, the bounded, symmetric set
\[V \defeq \overline{\{\pm M J u \colon \|u\|_{W^{1,n}}^n \le 1\}} \supset \overline{\{J u \colon \|u\|_{W^{1,n}}^n \le C \|J u\|_{\mathcal{H}^1} \le M}\}\]
contains a ball. It follows that $\overline{\text{co}}(V)$ contains a ball centered at the origin. By Lemma \ref{Main lemma}, $\{\sum_{j=1}^\infty M \lambda^j J u^j \colon \|u^j\|_{W^{1,n}} \le 1 \text{ for every } j \in \N, \; \sum_{j=1}^\infty |\lambda^j| < \infty\}$ is not of the first category, contradicting Theorem \ref{Main theorem}.

For the second claim recall that $J(\W)$ is dense in $\HR$. The claim follows from the first part of the proposition.
\end{proof}

Note that the proofs of Theorem \ref{Main theorem} and Proposition \ref{Proposition on failure of norm control} only use (implicitly) the infinity behavior of mappings in $\W$ and do not put restraints on the local behavior of solutions of the Jacobian equation. In particular, the arguments do not apply in bounded domains or to mappings in $\dot{W}^{1,n}(\R^n,\R^n)$.

\section{Theorem \ref{Bilinear theorem} and its variants} \label{Proof of Theorem}

The equivalence between conditions (i)--(ii) of Theorem \ref{Bilinear theorem} is proved in the case $A_1 = 0$ in \textsection \ref{Proof of equivalence (i) (ii)}, and the case $A_1 \in \{\dive,\curl\}$ is covered in \textsection \ref{The case A1 in div, curl}. The nonsurjectivity claim of Theorem \ref{Bilinear theorem} is shown in \textsection \ref{Nonsurjectivity}. In \textsection \ref{A further equivalent condition in the case A1 = A2 = 0} we present another equivalent condition in the case $A_1 = A_2 = 0$, and an analogue of Theorem \ref{Bilinear theorem} is discussed in \textsection \ref{A variant for homogeneous Sobolev spaces}.

\subsection{Proof of equivalence (i) $\Leftrightarrow$ (ii) when $A_1 = 0$} \label{Proof of equivalence (i) (ii)}
We start by explaining the main ideas behind the proof of implication (i) $\Rightarrow$ (ii) when $A_1 = 0$. The converse is, of course, classical. The argument is based loosely on P. Strzelecki's proof of ~\cite[Theorem 2.1]{Str01} (Corollary \ref{Cor Strzelecki} in this article).

We fix $x \in \R^n$ and $t > 0$ aspiring to estimate
\[\eta_t * \B(u,v)(x) = \sum_{i=1}^{m_1} \sum_{j=1}^{m_2} \sum_{|\alpha| \le k_1, |\beta| \le k_2} c_{i,j,\alpha,\beta} \int_{\R^n} \eta_t(x-\cdot) \partial^\alpha u_i \partial^\beta v_j.\]
As in e.g. ~\cite{CLMS93} and ~\cite{Str01}, given $i,j,\alpha,\beta$ we wish to move some orders of differentiation from $u_i$ onto the test function $\eta_t(x-\cdot)$. To this end, we use the Leibniz rule to write $\partial^\alpha (\eta_t(x-\cdot) u_i) = \sum_{\gamma \le \alpha} {\alpha \choose \gamma} \partial^{\alpha-\gamma} [\eta_t(x-\cdot)] \partial^\gamma u_i$ and assumption (ii) to get $\int_{\R^n} \B(\eta_t(x-\cdot) u,v) = 0$. In the estimation of the integrals $\int_{\R^n} \partial^{\alpha-\gamma} \eta_t(x-\cdot) \partial^\gamma u_i \partial^\beta v_j$ we want to use the higher order Poincar\'{e} inequality, Lemma \ref{Higher order Poincare lemma}, on the factors $\partial^\gamma u_i$ in order to cancel the negative powers of $t$ that the differentiation of $\eta_t(x-\cdot)$ creates. In order for the Poincar\'{e} inequality to be applicable we need to subtract a suitable polynomial from $u$ before using the Leibniz rule.

At this point the obstacle is that even though given any polynomial $P \colon \R^n \to \R^{m_1}$ we have $\int_{\R^n} \B(\eta_t(x-\cdot) (u-P),v) = 0$, it is not guaranteed that the use of the Leibniz rule
\[\partial^\alpha [\eta_t(x-\cdot) (u_i-P_i)] = \sum_{\gamma \le \alpha} {\alpha \choose \gamma} \partial^{\alpha-\gamma} [\eta_t(x-\cdot)] \partial^\gamma (u_i-P_i)\]
is successful in every term. Indeed, if $|\alpha| > \deg P_i + 1$, then $P_i$ is in general unsuitable for use in the Poincar\'{e} inequality when $|\gamma|=|\alpha|-1$, since $\partial^\gamma P_i = 0$. On the other hand, if $|\alpha| \le \deg P_i$, then the term with $\gamma = \alpha$ produces unwanted extra terms.

As a consequence, we first need to manipulate $\B$ in such a way that $|\alpha|$ can be taken to be the same in every term of $\B$ (and the degree of $P$ is $|\alpha|-1$). We accomplish this in 
Lemmas \ref{B decomposition lemma} and \ref{Step lemma}, and thereby the higher order Poincar\'{e} inequality becomes available.

\begin{lemma} \label{B decomposition lemma}
Suppose $\B$ satisfies condition (i) of Theorem \ref{Bilinear theorem}. Then, for every $l \le k_1+k_2$ and every $(u,v) \in W^{k_1,p_1}(\R^n,\R^{m_1}) \times W_{A_2}^{k_2,p_2}(\R^n,\R^{m_2})$, the integral of
\begin{equation} \label{Decomposition of derivatives}
\B_l(u,v) \defeq \sum_{i,j} \sum_{|\alpha + \beta| = l} c_{i,j,\alpha,\beta} \partial^\alpha u_i \partial^\beta v_j
\end{equation}
vanishes. Furthermore, $\B_0 = 0$.
\end{lemma}

\begin{proof}
Let $(u,v) \in W^{k_1,p_1}(\R^n,\R^{m_1}) \times W^{k_2,p_2}_{A_2}(\R^n,\R^{m_2})$. When $\tau > 0$, note that $A_2[(v(\tau \cdot)] = \tau^q A_2 v(\tau \cdot) = 0$. Thus, by the assumption on $\B$ and a change of variables,
\[0 = \int_{\R^n} \B(u(\tau \cdot),v(\tau \cdot)) = \tau^{-n} \sum_{l=0}^{k_1+k_2} \tau^l \int_{\R^n} \B_l(u,v).\]
By letting $\tau$ vary we get $\int_{\R^n} \B_l(u,v) = 0$ for every $l$.

Since the integral $\int_{\R^n} \B_0(u,v) = \sum_{i,j} c_{i,j,0,0} \int_{\R^n} u_i v_j$ vanishes for every $(u,v) \in W^{k_1,p_1}(\R^n,\R^{m_1}) \times W_{A_2}^{k_2,p_2}(\R^n,\R^{m_2})$, we conclude that $\sum_j c_{i,j,0,0} v_j = 0$ for every $i \in \{1,\dots,m_1\}$ and every $v \in W_{A_2}^{k_2,p_2}(\R^n,\R^{m_2})$, and the claim $\B_0 = 0$ follows.
\end{proof}

We now use \eqref{Decomposition of derivatives} to write $\B = \sum_{l=1}^{k_1+k_2} \B_l$ and treat the operators $\B_l$ separately. The following lemma allows us to further decompose the operators $\B_l$ so that in each new operator the value $|\alpha|$ is the same in every term.

\begin{lemma} \label{Step lemma}
Suppose $\B$ satisfies condition (i) of Theorem \ref{Bilinear theorem}. Then we may write $\B = \sum_{l_1=1}^{k_1} \sum_{l_2=0}^{k_2} \B_{l_1,l_2} + R$, where
for every $l_1 \in \{1,\dots,k_1\}$, $l_2 \in \{0,\dots,k_2\}$, and $(u,v) \in W^{k_1,p_1}(\R^n,\R^{m_1}) \times W_{A_2}^{k_2,p_2}(\R^n,\R^{m_2})$ the bilinear quantity
\begin{equation} \label{B l1 l2}
\B_{l_1,l_2}(u,v) \defeq \sum_{i,j} \sum_{|\alpha|=l_1} \sum_{|\beta| = l_2} \tilde{c}_{i,j,\alpha,\beta} \partial^\alpha u_i \partial^\beta v_j
\end{equation}
satisfies $\int_{\R^n} \B_{l_1,l_2}(u,v) = 0$ and $\|R(u,v)\|_{\mathcal{H}^1} \lesssim \|u\|_{W^{k_1,p_1}} \|v\|_
{W^{k_2,p_2}}$. Furthermore, if $c_{i,j,0,\beta} = 0$ for all $i,j,\beta$, then $\|R(u,v)\|_{\mathcal{H}^1} \lesssim \sum_{l=1}^{k_1} \|\nabla^l u\|_{L^{p_1}} \|v\|_{W^{k_2,p_2}}$.
\end{lemma}

\begin{proof}
The idea of the proof is to show the following estimate: when $i \in \{1,\dots,m_1\}$, $j \in \{1,\dots,m_2\}$, $k \in \{1,\dots,n\}$, $|\gamma| < k_1$ and $|\theta| < k_2$, we have
\begin{equation} \label{Step estimate}
\|\partial_k(\partial^\gamma u_i \partial^\theta v_j)\|_{\mathcal{H}^1} \lesssim \|u_i\|_{W^{k_1,p_1}} \|v_j\|_{W^{k_2,p_2}}.
\end{equation}
Before proving \eqref{Step estimate} we note that if the sum defining $\B_l$ in \eqref{Decomposition of derivatives} contains terms $c_{i,j,\alpha,\beta} \partial^\alpha u_i \partial^\beta v_j$ and $c_{i,j,\alpha',\beta'} \partial^{\alpha'} u_i \partial^{\beta'} v_j$ with $|\alpha+\beta| = |\alpha'+\beta'|$ and $|\alpha| < |\alpha'|$, we may use \eqref{Step estimate} to replace $c_{i,j,\alpha,\beta} \partial^\alpha u_i \partial^\beta v_j$ by $-c_{i,j,\alpha,\beta} \partial^{\alpha+e_k} u_i \, \partial^{\beta-e_k} v_j$ for some $k \in \{1,\dots,n\}$ (and add $c_{i,j,\alpha,\beta} \partial_k (\partial^\alpha u_i \partial^{\beta-e_k} v_j)$ to $R$). The claim then follows by induction.

In order to prove \eqref{Step estimate} let $x \in \R^n$ and $t > 0$. Choose a cube $Q \supset B(x,t)$ with $|Q| \lesssim_n t^n$. We fix $\epsilon \in ]0,p_2-1[$ choosing $\epsilon$ such that $1^* < (p_2-\epsilon)'$ (and furthermore $(p_2-\epsilon)' < p_1^*$ if $p_1 < n$) and define $\delta > 0$ by $(p_1-\delta)^* = (p_2-\epsilon)'$. An integration by parts and H\"{o}lder's and Poincar\'{e}'s inequalities yield
\[\begin{array}{lcl}
& &   \ds \int_{\R^n} \eta_t(x-\cdot) \partial_k (\partial^\gamma u_i \partial^\theta v_j) \\
&=&   \ds - \int_Q \partial_k[\eta_t(x-\cdot)] [(\partial^\gamma u_i - (\partial^\gamma u_i)_Q) \partial^\theta v_j + (\partial^\gamma u_i)_Q (\partial^\theta v_j - (\partial^\theta v_j)_Q)] \\
&\lesssim& \ds \frac{1}{t} \left( \Aint_Q |\partial^\gamma u_i  - (\partial^\gamma u_i)_Q|^{(p_2-\epsilon)'} \right)^\frac{1}{(p_2-\epsilon)'} \left( \Aint_Q |\partial^\theta v_j|^{p_2-\epsilon} \right)^\frac{1}{p_2-\epsilon} \\
&+&        \ds \frac{1}{t} |(\partial^\gamma u_i)_Q| \Aint_Q |(\partial^\theta v_j - (\partial^\theta v_j)_Q)| \\
&\lesssim& \ds \sum_{|\alpha| \le k_1}  M_{p_1-\delta} \left( |\partial^\alpha u_i| \right)(x) \sum_{|\beta| \le k_2} M_{p_2-\epsilon} \left( |\partial^\beta v_j| \right)(x).
\end{array}\]
Now \eqref{Step estimate} follows by using H\"{o}lder's inequality and the Hardy-Littlewood Maximal Theorem.
\end{proof}

The proof of direction (i) $\Rightarrow$ (ii) in the case $A_1 = 0$ will be completed by applying the following lemma. Note that Lemma \ref{Lemma for i to ii} readily generalizes to multilinear operators; this fact is recorded in Theorem \ref{Homogeneous bilinear theorem}. In the formulation of Lemma \ref{Lemma for i to ii} the assumptions and conclusion are in mismatch in order for the lemma to be usable in \textsection \ref{The case A1 in div, curl}.

\begin{lemma} \label{Lemma for i to ii}
Suppose $l_1 \in \{1,\dots,k_1\}$, $l_2 \in \{0,\dots,k_2\}$ and $\B_{l_1,l_2}$ is of the form \eqref{B l1 l2}. If $\int_{\R^n} \B_{l_1,l_2}(u,v) = 0$ for every $(u,v) \in W^{k_1,p_1}(\R^n,\R^{m_1}) \times W^{k_2,p_2}_{A_2}(\R^n,\R^{m_2})$, then $\|\B_{l_1,l_2}(u,v)\|_{\mathcal{H}^1} \lesssim  \|\nabla^{l_1} u\|_{L^{p_1}} \|\nabla^{l_2} v\|_{L^{p_2}}$ for every $(u,v) \in \dot{W}^{l_1,p_1}(\R^n,\R^{m_1}) \times W^{k_2,p_2}_{A_2}(\R^n,\R^{m_2})$.
\end{lemma}

\begin{proof}
Let $(u,v) \in \dot{W}^{l_1,p_1}(\R^n,\R^{m_1}) \times W_{A_2}^{k_2,p_2}(\R^n,\R^{m_2})$ and fix $x \in \R^n$ and $t > 0$. Select a cube $Q \supset B(x,t)$ with $|Q| \lesssim_n |B(x,t)|$ and use Lemma \ref{Higher order Poincare lemma} to select a polynomial $P^{l_1-1}_Q$. By assumption $\int_{\R^n} \B_{l_1,l_2}(\eta_t (u-P^{l_1-1}_Q u),v) = 0$, and so the Leibniz rule gives
\[\begin{array}{lcl}
& & \eta_t * \B_{l_1,l_2}(u,v)(x) \\
&=& \ds \sum_{i,j} \sum_{|\alpha|=l_1,|\beta|=l_2} \sum_{\gamma < \alpha} c_{i,j,\alpha,\beta,\gamma} \int_{\R^n} \partial^{\alpha-\gamma} [\eta_t(x-\cdot)] \partial^\gamma (u_i- P^{l_1-1}_Q u_i) \partial^\beta v_j.
  \end{array}\]
With $i,j,\alpha,\beta$ and $\gamma$ fixed, we choose again $\epsilon \in ]0,p_2-1[$ such that $1^* < (p_2-\epsilon)'$ (and $(p_2-\epsilon)' < p_1^*$ if $p_1 < n$) and define $\delta > 0$ by $(p_1-\delta)^* = (p_2-\epsilon_2)'$. By H\"{o}lder's inequality and Lemma \ref{Higher order Poincare lemma},
\[\begin{array}{lcl}
& &   \ds \int_{\R^n} \partial^{\alpha-\gamma} [\eta_t(x-\cdot)] \partial^\gamma (u_i- P^{l_1-1}_Q u_i) \partial^\beta v_j \\
&\lesssim& \ds t^{-|\alpha-\gamma|} \Aint_Q |\partial^\gamma (u_i - P^{l_1-1}_Q u_i) \partial^\beta v_j| \\
&\lesssim& \ds t^{-|\alpha-\gamma|} \left( \Aint_Q |\partial^\gamma (u_i - P^{l_1-1}_Q u_i)|^{(p_2-\epsilon)'} \right)^\frac{1}{(p_2-\epsilon)'} \left( \Aint_Q |\partial^\beta v_j|^{p_2-\epsilon} \right)^\frac{1}{p_2-\epsilon} \\
&\lesssim& \ds M_{p_1-\delta}(\nabla^{l_1} u_i) M_{p_2-\epsilon}(\partial^\beta v_j).
  \end{array}\]
The claimed estimate follows from H\"{o}lder's inequality and the Hardy-Littlewood Maximal Theorem.
\end{proof}

We have now showed the equivalence of conditions (i) and (ii) of Theorem \ref{Bilinear theorem} in the case $A_1 = 0$. In the proof the identity $\int_{\R^n} \B(\eta_t(x-\cdot) u, v)  = 0$ allowed us to pass partial derivatives from $u$ onto $\eta_t(x-\cdot)$ by using the Leibniz rule. Such a technique is not available when $A_1 \in \{\dive,\curl\}$, as it is not guaranteed that $\eta_t(x-\cdot) u \in W_{A_1}^{k_1,p_1}(\R^n,\R^n)$. We overcome this obstacle in \textsection \ref{The case A1 in div, curl} by using potentials for divergence and curl.

\subsection{The case $A_1 \in \{\dive,\curl\}$} \label{The case A1 in div, curl}
Direction (ii) $\Rightarrow$ (i) is classical as before, and the challenge is direction (i) $\Rightarrow$ (ii). Let $A_1 = \dive$; the case $A_1 = \curl$ is handled similarly.

Let $u \in W_{\dive}^{k_1,p_1}(\R^n,\R^n)$ and $v \in W_{A_2}^{k_2,p_2}(\R^n,\R^n)$. Using the potential $\Phi$ given by Proposition \ref{Potential proposition} and identifying $\R ^{n \times n}_\Delta$ with $\R^{n(n-1)/2}$ we write
\[\begin{array}{lcl}
    \B(u,v)
&=& \ds \sum_{i=1}^n \sum_{j=1}^{m_2} \sum _{|\alpha| \le k_1, |\beta| \le k_2} c_{i,j,\alpha,\beta} \partial^\alpha u_i \partial^\beta v_j \\
&=& \ds \sum_{i=1}^{n(n-1)/2} \sum_{j=1}^{m_2} \sum_{1 \le |\alpha| \le k_1+1, |\beta| \le k_2} \tilde{c}_{i,j,\alpha,\beta} \partial^\alpha \Phi_i \partial^\beta v_j \\
&\eqdef& \ds \tilde{\mathscr{B}}(\Phi,v).
  \end{array}\]
Using Lemmas \ref{Step lemma} and \ref{Lemma for i to ii} we may further decompose $\tilde{\mathscr{B}} = \sum_{l_1=1}^{k_1+1} \sum_{l_2=0}^{k_2} \tilde{\mathscr{B}}_{l_1,l_2} + R$, where
\[\tilde{\mathscr{B}}_{l_1,l_2}(\Phi,v) \defeq \sum_{i=1}^{n(n-1)/2} \sum_{j=1}^{m_2} \sum_{|\alpha| = l_1, |\beta| = l_2} \tilde{c}_{i,j,\alpha,\beta} \partial^\alpha \Phi_i \partial^\beta v_j\]
for all $l_1,l_2$ and $\|R(\Phi,v)\|_{\mathcal{H}^1} \lesssim \sum_{l=1}^{k_1} \|\nabla^l \Phi\|_{L^{p_1}} \|v\|_{W^{k_2,p_2}}$. By Lemma \ref{Lemma for i to ii} and Proposition \ref{Potential proposition 2},
\[\|\tilde{\mathscr{B}}(\Phi,v)\|_{\mathcal{H}^1} \lesssim \sum_{l=1}^{k_1} \|\nabla^l \Phi\|_{L^{p_1}} \|v\|_{W^{k_2,p_2}} \lesssim \|u\|_{W^{k_1,p_1}} \|v\|_{W^{k_2,p_2}}.\]
This completes the proof of implication (i) $\Rightarrow$ (ii) when $A_1 = \dive$.

\subsection{Nonsurjectivity} \label{Nonsurjectivity}
We assume $A_1 = \curl$, the case $A_1 = \dive$ being similar. The proof for $A_1 = 0$ is essentially a special case of the proof of Theorem \ref{Extension to homogeneous polynomials}. In the argument given below we tacitly assume that $A_1 u^l = 0$ and $A_2 v^l = 0$ for every $l \in \N$.

\begin{proof}[Completion of the proof of Theorem \ref{Bilinear theorem}]
Seeking contradiction, assume that the set
\[\begin{array}{lcl}
& & \ds \left\{ \sum_{l=1}^\infty \lambda^l \B(u^l,v^l) \colon \sup_{l \in \N} \|u^l\|_{W^{k_1,p_1}} \|v^l\|_{W^{k_2,p_2}} \le 1, \; \sum_{l=1}^\infty |\lambda^l| < \infty \right\} \\
&=& \ds \left\{ \sum_{l=1}^\infty \lambda^l \B(u^l,v^l) \colon \sup_{l \in \N} (\|u^l\|_{W^{k_1,p_1}}^2 + \|v^l\|_{W^{k_2,p_2}}^2) \le 1, \; \sum_{l=1}^\infty |\lambda^l| < \infty \right\}
  \end{array}\]
is of the second category. Set $V \defeq \left\{ \B(u,v) \colon \|u\|_{W^{k_1,p_1}}^2 + \|v\|_{W^{k_2,p_2}}^2 \le 1 \right\}$. Now condition (iv) of Lemma \ref{Main lemma} is satisfied.

As in the proof of Theorem \ref{Main theorem}, fix $b \in \CO \setminus \{0\}$. Use condition (i) of Lemma \ref{Main lemma} to select, for every $\tau > 0$, a mapping $(\tilde{u}
,\tilde{v}) = (u(\tau \cdot),v(\tau \cdot)) \in (W^{k_1,p_1}_{\curl}(\R^n,\R^n) \times W_{A_2}^{k_2,p_2}(\R^n,\R^{m_2})) \setminus \{0\}$ satisfying
\begin{equation} \label{False norm estimate}
\|b\|_{\text{BMO}} = \|b(\tau \cdot)\|_{\text{BMO}} \lesssim \frac{\displaystyle \int_{\R^n} b(\tau \cdot) \B(u(\tau \cdot),v(\tau \cdot))}{\|u(\tau \cdot)\|_{W^{k_1,p_1}}^2 + \|v(\tau \cdot)\|_{W^{k_2,p_2}}^2}.
\end{equation}
We will show that \eqref{False norm estimate} leads to a contradiction. We use Proposition \ref{Potential proposition 3} to write $u = \nabla \Psi$, where $\Psi \in \cap_{l=1}^{k_1+1} \dot{W}^{l,p_1}(\R^n)$, so that $\partial^\alpha [u_i(\tau \cdot)] = \tau^{-1} \partial^{\alpha+e_i} [\Psi(\tau \cdot)]$ for all $\alpha \in (\N \cup \{0\})^n$ and $i \in \{1,\dots,n\}$. We choose a cube $Q \supset \supp(b)$ and assume, without loss of generality, that $\int_Q \Psi = 0$.

We use the Leibniz rule (as before) and a change of variables to write
\[\begin{array}{lcl}
& &               \ds \int_{\R^n} b(\tau \cdot) \B(u(\tau \cdot),v(\tau \cdot)) \\
&=& \ds \tau^{-1} \sum_{i,j} \sum_{\alpha,\beta} c_{i,j,\alpha,\beta} \int_{\R^n} b(\tau \cdot) \partial^{\alpha+e_i} [\Psi(\tau \cdot)] \partial^\beta [v_j(\tau \cdot)] \\
&=& \displaystyle \tau^{-1} \sum_{i,j} \sum_{\alpha,\beta} \sum_{\gamma < \alpha+e_i} c_{i,j,\alpha,\beta,\gamma} \int_{\R^n} \partial^{\alpha+e_i-\gamma} [b(\tau \cdot)] \partial^\gamma [\Psi(\tau \cdot)] \partial^\beta [v_j(\tau \cdot)] \\
&=& \displaystyle \sum_{i,j} \sum_{\alpha,\beta} \sum_{\gamma < \alpha+e_i} c_{i,j,\alpha,\beta,\gamma} \tau^{-n-1} \int_{\R^n} \tau^{|\alpha+e_i|} \partial^{\alpha+e_i-\gamma} b \, \partial^\gamma \Psi \tau^{|\beta|} \partial^\beta v_j \\
&\lesssim& \displaystyle \sum_{i,j} \sum_{\alpha,\beta} \sum_{\gamma < \alpha+e_i} \tau^{-n-1} \int_{\R^n} \left| \tau^{|\alpha+e_i|} \partial^{\alpha+e_i-\gamma} b \, \partial^\gamma \Psi \tau^{|\beta|} \partial^\beta v_j \right|.
  \end{array}\]
We control single terms where $|\gamma| \ge 1$ by using H\"{o}lder's and Young's inequalities:
\[\begin{array}{lcl}
& &   \displaystyle \tau^{-n-1} \int_{\R^n} \left| \tau^{|\alpha+e_i|} \partial^{\alpha+e_i-\gamma} b \, \partial^\gamma \Psi \tau^{|\beta|} \partial^\beta v_j \right| \\
&\hspace{0.13cm} \lesssim_b& \displaystyle \tau^{-1} M \tau^{|\alpha+e_i|-\frac{n}{p_1}} \|\partial^\gamma \Psi\|_{L^{p_1}} \frac{\tau^{|\beta|-\frac{n}{p_2}}}{M} \|\partial^\beta v_j\|_{L^{p_2}} \\
&\le& \displaystyle \tau^{-2} \frac{\left( M \tau^{|\alpha+e_i|-\frac{n}{p_1}} \right)^2}{2} \|\partial^\gamma \Psi\|_{L^{p_1}}^2 + \frac{\left( {\tau^{|\beta|-\frac{n}{p_2}}} \right)^2}{2 M^2}\|\partial^\beta v_j\|_{L^{p_2}}^2.
  \end{array}\]
The terms with $\gamma = 0$ are, in turn, estimated by using the Poincar\'{e} inequality on $\Psi$ in $Q$:
\[\begin{array}{lcl}
& &   \displaystyle \tau^{-n-1} \int_{\R^n} \left| \tau^{|\alpha+e_i|} \partial^{\alpha+e_i} b \, \Psi \, \tau^{|\beta|} \partial^\beta v_j \right| \\
&\hspace{0.13cm} \lesssim_b& \displaystyle \tau^{-1} M \tau^{|\alpha+e_i|-\frac{n}{p_1}} \|\Psi\|_{L^{p_1}(Q)} \frac{\tau^{|\beta|-\frac{n}{p_2}}}{M} \|\partial^\beta v_j\|_{L^{p_2}} \\
&\hspace{0.13cm} \lesssim_b& \displaystyle \tau^{-2} \frac{\left( M \tau^{|\alpha+e_i|-\frac{n}{p_1}} \right)^2}{2} \|\nabla \Psi\|_{L^{p_1}(Q)}^2 + \frac{\left( {\tau^{|\beta|-\frac{n}{p_2}}} \right)^2}{2 M^2}\|\partial^\beta v_j\|_{L^{p_2}}^2.
  \end{array}\]

By combining the inequalities and using \eqref{False norm estimate} we conclude that
\[\begin{array}{lcl}
& & \displaystyle \sum_{i,j} \sum_{\alpha,\beta} \sum_{0 < \gamma < \alpha+e_i} \left( \tau^{-2} \left( \tau^{p_1 |\gamma| - n} \int_{\R^n} |\partial^\gamma \Psi|^{p_1} \right)^\frac{2}{p_1} \right. \\
&+& \ds \left. \left( \tau^{p_2 |\beta| - n} \int_{\R^n} |\partial^\beta v_j|^{p_2} \right)^\frac{2}{p_2} \right) \\
&=& \displaystyle \sum_{i,j} \sum_{\alpha,\beta} \sum_{0 < \gamma < \alpha+e_i} \left( \tau^{-2} \left( \int_{\R^n} |\partial^\gamma [\Psi(\tau \cdot)]|^{p_1} \right)^\frac{2}{p_1} \right. \\
&+& \ds \left. \left( \int_{\R^n} |\partial^\beta [v_j(\tau \cdot)]|^{p_2} \right)^\frac{2}{p_2} \right) \\
&\lesssim& \displaystyle \tau^{-2} \sum_{l=1}^{k_1+1} \|\nabla^l [\Psi(\tau \cdot)]\|_{L^{p_1}}^2 +  \|v(\tau \cdot)\|_{W^{k_2,p_2}}^2 \\
&\lesssim& \displaystyle \|u(\tau \cdot)\|_{W^{k_1,p_1}}^2 +  \|v(\tau \cdot)\|_{W^{k_2,p_2}}^2 \\
&\hspace{0.14cm} \lesssim_{b}& \displaystyle \int_{\R^n} b(\tau \cdot) \B(u(\tau \cdot),v(\tau \cdot)) \\
&\hspace{0.14cm} \lesssim_{b}& \displaystyle \sum_{i,j} \sum_{\alpha,\beta} \sum_{0 < \gamma < \alpha+e_i} \left( \tau^{-2} \left( M^{p_1} \tau^{p_1 |\alpha+e_i| - n} \int_{\R^n} |\partial^\gamma \Psi|^{p_1} \right)^\frac{2}{p_1} \right. \\
&+& \ds \left. \left( \frac{\tau^{p_2 |\beta| - n}}{M^{p_2}} \int_{\R^n} |\partial^\beta v_j|^{p_2} \right)^\frac{2}{p_2} \right).
  \end{array}\]
We get a contradiction by using the assumption $|\alpha+e_i| > |\gamma|$ for all multi-indices in the sums, choosing $M$ large enough and then choosing $\tau$ small enough.
\end{proof}

\subsection{A further equivalent condition in the case $A_1 = A_2 = 0$} \label{A further equivalent condition in the case A1 = A2 = 0}

While condition (i) of Theorem \ref{Bilinear theorem} tends to be easy to verify, it is not always completely transparent. In the following result we give, in the case $A_1 = A_2 = 0$, another equivalent condition that can be tested by simple examination of the coefficients of $\B$.

\begin{prop} \label{Coefficient proposition}
Suppose $\B$ satisfies the assumptions of Theorem \ref{Bilinear theorem} and $A_1 = A_2 = 0$. Then the following conditions are equivalent.

\renewcommand{\labelenumi}{(\roman{enumi})}
\begin{enumerate}

\item For every $\gamma \in (\N \cup \{0\})^n$ and every $i \in \{1,\dots,m_1\}$ and $j \in \{1,\dots,m_2\}$,
\begin{equation} \label{Condition on coefficients}
\sum_{\alpha+\beta=\gamma} (-1)^{|\alpha|} c_{i,j,\alpha,\beta} = 0.
\end{equation}

\item For every $(u,v) \in W^{k_1,p_1}(\R^n,\R^{m_1}) \times W^{k_2,p_2}(\R^n,\R^{m_2})$,
\[\|\B(u,v)\|_{\mathcal{H}^1} \lesssim \|u\|_{W^{k_1,p_1}} \|v\|_{W^{k_2,p_2}}.\]
\end{enumerate}
\end{prop}

\begin{proof}
When (i) holds and $(u,v) \in C_c^\infty(\R^n,\R^{m_1}) \times C_c^\infty(\R^n,\R^{m_2})$, we compute
\[\begin{array}{lcl}
    \ds \int_{\R^n} \B(u,v)
&=& \ds \sum_{\gamma} \sum_{i,j} \sum_{\alpha+\beta=\gamma} c_{i,j,\alpha,\beta} \int_{\R^n} \partial^\alpha u_i \partial^\beta v_j \\
&=& \ds \sum_\gamma \sum_{i,j} \left( \int_{\R^n} u_i \partial^{\gamma} v_j \sum_{\alpha+\beta=\gamma} (-1)^{|\alpha|} c_{i,j,\alpha,\beta} \right) = 0.
  \end{array}\]
By continuity, $\int_{\R^n} \B(u,v) = 0$ for all $(u,v) \in W^{k_1,p_1}(\R^n,\R^{m_1}) \times W^{k_2,p_2}(\R^n,\R^{m_2})$, and Theorem \ref{Bilinear theorem} then implies condition (ii).

For direction (ii) $\Rightarrow$ (i) we fix $\gamma \in (\N \setminus \{0\})^n$ and let $(u,v) \in C_c^\infty(\R^n,\R^{m_1+m_2})$. By using an integration by parts we write, as above,
\[0
= \int_{\R^n} \B(u,v) \\
= \sum_i \left( \sum_j \sum_\gamma \sum_{\alpha+\beta=\gamma} (-1)^{|\alpha|} c_{i,j,\alpha,\beta} \int_{\R^n} u_i \partial^{\gamma} v_j \right).\]
When $i_0 \in \{1,\dots,m_1\}$, letting $u_{i_0} \in C_c^\infty(\R^n)$ and $u_i = 0$ for every $i \neq i_0$ we get $\sum_j \sum_\gamma \sum_{\alpha+\beta=\gamma} (-1)^{|\alpha|} c_{i_0,j,\alpha,\beta} \partial^{\gamma} v_j = 0$ for all $v \in C_c^\infty(\R^n,\R^{m_2})$. Fixing $\gamma \in (\N \cup \{0\})^n$ with $|\gamma| = k_1 + k_2$ and setting locally $v_{j_0} = x^\gamma$ and $v_j = 0$ for all $j \neq j_0$ we obtain $\sum_{\alpha+\beta=\gamma} (-1)^{|\alpha|} c_{i_0,j_0,\alpha,\beta} = 0$. Consequently,
\[\sum_i \left( \sum_j \sum_\gamma \sum_{|\alpha+\beta|<l} (-1)^{|\alpha|} c_{i,j,\alpha,\beta} \int_{\R^n} u_i \partial^{\gamma} v_j \right) = 0\]
holds for $l = k_1+k_2$. Continuing by downward induction on $l$ we obtain condition (i).
\end{proof}

\subsection{A variant for homogeneous Sobolev spaces} \label{A variant for homogeneous Sobolev spaces}
The aim of this subsection is to give an analogue of Theorem \ref{Bilinear theorem} for homogeneous Sobolev spaces when $|\alpha|=k_1$ and $|\beta|=k_2$ in all non-vanishing terms of $\B$. This case is much easier than the one studied in Theorem \ref{Bilinear theorem} since the higher-order Poincar\'{e} inequality is readily available. Furthermore, multilinear operators are essentially as easy to handle as bilinear ones since Lemma \ref{Step lemma} is not needed. We note in passing that terms satisfying $|\alpha| < k_1$ or $|\beta| < k_2$ can in some cases be treated in homogeneous Sobolev spaces by using Sobolev embeddings, but we do not pursue the matter here.

We consider multilinear operators $\M \colon \prod_{j=1}^r \dot{W}_{A_j}^{k_j,p_j}(\R^n,\R^{m_j}) \to L^1(\R^n)$ of the form
\begin{equation} \label{Form of a bilinear operator 2}
\mathcal{M}(u^1,\dots,u^r) = \sum_{\substack{i_1 \in \{1,\dots,m_1\} \\ \vdots \\ i_r \in \{1,\dots,m_r\}}} \sum_{\substack{|\alpha_1| = k_1 \\ \vdots \\ |\alpha^r| = k_r}} c_{i_1,\dots,i_r,\alpha^1,\dots,\alpha^r} \prod_{j=1}^r \partial^{\alpha^j} u^j_{i_j},
\end{equation}
where
\begin{equation} \label{Conditions on exponents 2}
k_j \ge 0, \; n \ge 2, \; m_j \ge 1, \text{ and } p_j \in ]1,\infty[ \text{ with  } \frac{1}{p_1} + \cdots + \frac{1}{p_r} = 1.
\end{equation}

The case $A_1 = \cdots = A_r = 0$ of the following proposition follows e.g. from ~\cite[Theorem 1]{Gra92} and the whole result seems to follow by using the ideas of ~\cite{Gra92} and the potentials for divergence and curl constructed in \textsection \ref{Potential for divergence and curl}. However, we feel that the explicit statement and fairly elementary proof of Theorem \ref{Homogeneous bilinear theorem} bring further unity and clarity to the subject.

\begin{theorem} \label{Homogeneous bilinear theorem}
Assume that \eqref{Form of a bilinear operator 2} and \eqref{Conditions on exponents 2} hold and $A_1 \in \{0,\dive,\curl\}$. The following conditions are equivalent.
\renewcommand{\labelenumi}{(\roman{enumi})}
\begin{enumerate}
\item For every $(u^1,\dots,u^r) \in \Pi_{j=1}^r \dot{W}_A^{k_j,p_j}(\R^n,\R^{m_j})$,
\[\int_{\R^n} \M(u^1,\dots,u^r) = 0.\]

\item For every $(u^1,\dots,u^r) \in \Pi_{j=1}^r \dot{W}_{A_j}^{k_j,p_j}(\R^n,\R^{m_j})$,
\[\|\M(u^1,\dots,u^r)\|_{\mathcal{H}^1} \lesssim \prod_{j=1}^r \|\nabla^{k_j} u^j\|_{L^{p_j}}.\]
\end{enumerate}
\noindent If $r=2$ and $A_1=A_2 = 0$, then (i) and (ii) are equivalent to the following condition:
\renewcommand{\labelenumi}{(\roman{enumi})}
\begin{enumerate}
\setcounter{enumi}{2}
\item For every $\gamma \in (\N \cup \{0\})^n$ and every $i \in \{1,\dots,m_1\}$ and $j \in \{1,\dots,m_2\}$,
\[\sum_{\alpha+\beta=\gamma} c_{i,j,\alpha,\beta} = 0.\]
\end{enumerate}
\end{theorem}

In the case $A_1 = 0$ we note again that if $k_1 = 0$, then $\M = 0$. The case $A_1 = 0$, $k_1 \ge 1$ is proved by a simple adaptation of the proof of Lemma \ref{Lemma for i to ii}, and the case $A_1 \in \{\dive,\curl\}$ is shown as in \textsection \ref{The case A1 in div, curl}. Furthermore, when $r = 2$ and $A_1=A_2$, direction (i) $\Rightarrow$ (iii) follows from Proposition \ref{Coefficient proposition} while (iii) $\Rightarrow$ (i) checked by an integration by parts.

\begin{remark}
When $r \ge 3$, condition (iii) of Theorem \ref{Homogeneous bilinear theorem} has the following analogue: for every $\gamma \in (\N \cup \{0\})^n$, all indices $i_j \in \{1,\dots,m_j\}$ and all multi-indices $\theta^2,\dots,\theta^r$ with $\sum_{j=2}^r \theta^j = \gamma$,
\begin{equation} \label{Multilinear ugly analogue}
\sum_{\substack{\alpha^1+\cdots+\alpha^r=\gamma,\\\alpha^1+\alpha^2 \ge \theta^2,\dots,\alpha^1+\alpha^r \ge \theta^r}} c_{i_1,\dots,i_r,\alpha^1,\dots,\alpha^r} \prod_{j=2}^r {\alpha^1 \choose \theta^j - \alpha^j} = 0.
\end{equation}
Condition \eqref{Multilinear ugly analogue} appears to be far too complicated to have practical use and so we omit the somewhat cumbersome proof of its equivalence to conditions (i) and (ii) of Theorem \ref{Homogeneous bilinear theorem}.
\end{remark}

\section{Applications to multilinear quantities} \label{Applications to bilinear quantities}
We present some corollaries of Theorems \ref{Bilinear theorem} and \ref{Homogeneous bilinear theorem} and Proposition \ref{Coefficient proposition}.

\subsection{Applications to classes of operators} \label{Applications to classes of operators}
We first use Theorem \ref{Bilinear theorem} and Proposition \ref{Coefficient proposition} to treat some general classes of bilinear partial differential operators. In the first result we study the product of two linear constant-coefficient partial differential operators. We define the operators $\PP \colon W^{k_1,p_1}(\R^n,\R^{m_1}) \to L^{p_1}(\R^n)$ and $\Q \colon W^{k_2,p_2}(\R^n,\R^{m_2}) \to L^{p_2}(\R^n)$ by
\begin{equation} \label{Definition of P and Q}
\PP(u) \defeq \sum_{i=1}^{m_1} \sum_{|\alpha| \le k_1} c_{i,\alpha} \partial^\alpha u_i \quad \text{and} \quad \Q(v) \defeq \sum_{j=1}^{m_2} \sum_{|\beta| \le k_2} d_{j,\beta} \partial^\beta v_j.
\end{equation}
The estimate we present can be strengthened in an obvious manner when all the terms of $\PP$ and/or $\Q$ have the same order. The result follows directly from Proposition \ref{Coefficient proposition}.

\begin{cor} \label{Corollary on product of P and Q}
Let $\PP$ and $\Q$ be defined by \eqref{Definition of P and Q}. The following conditions are equivalent.

\renewcommand{\labelenumi}{(\roman{enumi})}
\begin{enumerate}

\item $\|\PP(u) \Q(v)\|_{\mathcal{H}^1} \lesssim \|u\|_{W^{k_1,p_1}} \|v\|_{W^{k_2,p_2}}$ for all $u \in W^{k_1,p_1}(\R^n,\R^m)$ and $v \in W^{k_2,p_2}(\R^n,\R^m)$.

\item $\sum_{\alpha+\beta=\gamma} (-1)^{|\alpha|} c_{i,\alpha} d_{j,\beta} = 0$ for all $i,j$ and all $\gamma \in (\N \cup \{0\})^n$.
\end{enumerate}
\end{cor}

In the following three corollaries we fix $n,k \in \N$, assume that $p_1,p_2 \in ]1,\infty[$ satisfy $1/p_1+1/p_2=1$ and define $\PP \colon \mathcal{D}'(\R^n) \to \mathcal{D}'(\R^n)$ by
\[\PP(u) \defeq \sum_{|\alpha| \le k} c_\alpha \partial^\alpha u.\]

\begin{cor} \label{Cor on difference of partial derivatives}
The following conditions are equivalent.

\renewcommand{\labelenumi}{(\roman{enumi})}
\begin{enumerate}

\item $\|\PP(u) v - \PP(v) u\|_{\mathcal{H}^1} \lesssim \|u\|_{W^{k,p_1}} \|v\|_{W^{k,p_2}}$ for all $u \in W^{k,p_1}(\R^n)$ and $v \in W^{k,p_2}(\R^n)$.

\item $|\alpha|$ is even in all the nonzero terms of $\PP$.
\end{enumerate}
\end{cor}

\begin{cor} \label{Cor on sum of partial derivatives}
The following conditions are equivalent.

\renewcommand{\labelenumi}{(\roman{enumi})}
\begin{enumerate}

\item $\|\PP(u) v + \PP(v) u\|_{\mathcal{H}^1} \lesssim \|u\|_{W^{k,p_1}} \|v\|_{W^{k,p_2}}$ for all $u \in W^{k,p_1}(\R^n)$ and $v \in W^{k,p_2}(\R^n)$.

\item $|\alpha|$ is odd in all the nonzero terms of $\PP$.
\end{enumerate}
\end{cor}

\begin{cor} \label{Cor on partial derivatives of product}
The following conditions are equivalent.

\renewcommand{\labelenumi}{(\roman{enumi})}
\begin{enumerate}

\item $\|\PP(uv)\|_{\mathcal{H}^1} \lesssim \|u\|_{W^{k,p_1}} \|v\|_{W^{k,p_2}}$ for all $u \in W^{k,p_1}(\R^n)$ and $v \in W^{k,p_2}(\R^n)$.

\item $c_0 = 0$.
\end{enumerate}
\end{cor}

\begin{proof}[Proof of Corollaries \ref{Cor on difference of partial derivatives}--\ref{Cor on sum of partial derivatives}]
We write
\[\B(u,v) \defeq \PP(u) v - \PP(v) u =  \sum_{|\alpha| \le k} c_\alpha (v \partial^\alpha u - u \partial^\alpha v).\]
Condition (i) of Proposition \ref{Coefficient proposition} is satisfied if and only if $[(-1)^{|\gamma|} - 1] c_\gamma = 0$ for every $\gamma \in (\N \setminus \{0\})^n$, that is, if and only if condition (ii) of Corollary \ref{Cor on difference of partial derivatives} holds. Proposition \ref{Coefficient proposition} now implies Corollary \ref{Cor on difference of partial derivatives}. The proof of Corollary \ref{Cor on sum of partial derivatives} is almost identical.
\end{proof}

\begin{proof}[Proof of Corollary \ref{Cor on partial derivatives of product}]
The direction (i) $\Rightarrow$ (ii) is obvious. For the other direction we verify condition (ii) of Theorem \ref{Bilinear theorem} by fixing $\alpha$ with $|\alpha| \ge 1$ and noting that $\int_{\R^n} \partial^\alpha (uv) = 0$ for all $(u,v) \in W^{k,p_1}(\R^n) \times W^{k,p_2}(\R^n)$.
\end{proof}

\subsection{$\mathcal{H}^1$ regularity of specific operators} \label{H1 regularity of specific operators}
In this subsection we note that Theorem \ref{Homogeneous bilinear theorem} directly implies $\mathcal{H}^1$ regularity with a natural norm estimate for many familiar compensated compactness quantities. In fact, the elementary cases $A_1 \in \{0,\curl\}$ (which do not require the use of the somewhat technical potentials for divergence-free fields) suffice in every example presented here. Almost all of the results are from the original milestone paper ~\cite{CLMS93}.

\begin{cor}[~\cite{CLMS93}] \label{Prototypical corollary}
Suppose $p, q \in ]1, \infty[$ satisfy $1/p+1/q=1$ and let $v \in L^p_{\dive}(\R^n,\R^n)$ and $u \in \dot{W}^{1,q}(\R^n)$. Then
\begin{equation} \label{Prototypical estimate}
\|v \cdot \nabla u\|_{\mathcal{H}^1} \lesssim \|v\|_{L^p} \|\nabla u\|_{L^q}.
\end{equation}
\end{cor}

\begin{proof}
Since $\int_{\R^n} v \cdot \nabla u = - \langle \dive v, u \rangle = 0$ for all $v \in L^p_{\dive}(\R^n,\R^n)$ and $u \in C_c^\infty(\R^n)$, the claim follows from Theorem \ref{Homogeneous bilinear theorem} and the density of $C_c^\infty(\R^n)$ in $\dot{W}^{1,q}(\R^n)$.
\end{proof}

Corollary \ref{Prototypical corollary} easily implies the $\mathcal{H}^1$ regularity of the div-curl quantity since $\nabla \colon \dot{W}^{1,q}(\R^n) \to L^q_{\curl}(\R^n,\R^n)$ is a surjection.

\begin{cor}[~\cite{CLMS93}] \label{H1 regularity of div-curl quantity}
Suppose $E \in L^p_{\dive}(\R^n,\R^n)$ and $B \in L^q_{\operatorname{curl}}(\R^n,\R^n)$, where $p, q \in ]1, \infty[$ satisfy $1/p+1/q=1$. Then $\|E \cdot B\|_{\mathcal{H}^1} \lesssim \|E\|_{L^p} \|B\|_{L^q}$.
\end{cor}

Since $C_c^\infty(\R^n,\R^n)$ is dense in $\dot{W}^{1,2}(\R^n,\R^n)$, the following result follows from Theorem \ref{Homogeneous bilinear theorem} by an integration by parts.

\begin{cor}[~\cite{CLMS93}] \label{NS lemma}
Let $u \in \dot{W}^{1,2}(\R^n,\R^n)$ and $v \in \dot{W}^{1,2}_{\dive}(\R^n,\R^n)$. Then we have $\| \sum_{i,j=1}^n \partial_j u_i \partial_i v_j \|_{\mathcal{H}^1} \lesssim \|\D u\|_{L^2} \|\D v\|_{L^2}$.
\end{cor}

Corollaries \ref{Prototypical corollary} and \ref{NS lemma} have, among others, the following consequence on Leray solutions of the Navier-Stokes equations, where the formula $-\Delta p = \sum_{i,j=1}^3 \partial_j u_i \partial_i u_j$ is used.

\begin{cor}[~\cite{CLMS93}] \label{Navier-Stokes corollary}
Assume that a velocity field $u \in L^2(]0,\infty[; \dot{W}^{1,2}(\R^3,\R^3)) \cap L^\infty(]0,\infty[; L^2(\R^3,\R^3))$ and a pressure $p \in \mathcal{D}'(]0,\infty[ \times \R^3)$ satisfy the incompressible Navier-Stokes equations $\partial_t u + (u \cdot \nabla) u + \nabla p = \nu \Delta u$ and $\dive u = 0$ with viscosity $\nu > 0$. The convective term $(u \cdot \nabla) u$ belongs to $L^2(]0,\infty[; \mathcal{H}^1(\R^3))$ and $\Delta p $ belongs to $L^1(]0,\infty[; \mathcal{H}^1(\R^3))$.
\end{cor}

Theorem \ref{H1 regularity of Jacobians} is a direct corollary of Theorem \ref{Homogeneous bilinear theorem} since
\[\int_{\R^n} J u =\sum_{j=1}^n (-1)^j  \int_{\R^n} \partial_j \left( u_1 \frac{\partial(u_2,\dots,u_n)}{\partial(x_1,\dots,\hat{x_j},x_n)} \right) = 0\]
for every $u \in C_c^\infty(\R^n,\R^n)$. We also mention the following consequence of Theorem \ref{H1 regularity of Jacobians}.

\begin{cor}[~\cite{CLMS93}]
If $u \in \dot{W}^{2,n}(\R^n)$, then the Hessian $\mathbf{H} u \defeq \det[(\nabla \otimes \nabla) u]$ belongs to $\HR$ and $\|\mathbf{H} u\|_{{\mathcal H}^1} \lesssim \Vert \nabla^2 u \Vert_{L^n}^n$.
\end{cor}

The following generalization of Corollary \ref{Prototypical corollary} by P. Strzelecki and its proof in ~\cite{Str01} were among the main inspirations behind Theorem \ref{Bilinear theorem} and Theorem \ref{Homogeneous bilinear theorem}.

\begin{cor}[~\cite{Str01}] \label{Cor Strzelecki}
Let $u \in \dot{W}^{k,p}(\R^n)$, where $k \ge 1$ and $1 < p < \infty$. Suppose $E = (E_\alpha)_{|\alpha|=k} \in L^{p'}(\R^n,\R^{kn})$ satisfies $\nabla^k \cdot E = 0$, that is, $\sum_{|\alpha|=k} \int_{\R^n} E_\alpha \partial^\alpha \varphi = 0$ for all $\varphi \in C_c^\infty(\R^n)$. Then $\|\sum_{|\alpha|=k} E_\alpha \partial^\alpha u\|_{\mathcal{H}^1} \lesssim \|E\|_{L^{p'}} \|\nabla^k u\|_{L^p}$.
\end{cor}

The last compensated compactness quantity we mention appears in the von K\'{a}rm\'{a}n equations (~\cite[p. 261]{CLMS93}, see also ~\cite{CM15} and the references contained therein). Condition (iii) of Theorem \ref{Homogeneous bilinear theorem} is clearly satisfied.

\begin{cor}[~\cite{CLMS93}]
When $(u, v) \in \dot{W}^{2,2}(\R^2,\R^2)$, the Monge-Amp\`{e}re form
$[u,v] \defeq u_{xx} v_{yy} + u_{yy} v_{xx} - 2 u_{xy} v_{xy}$ satisfies $\|[u,v]\|_{\mathcal{H}^1} \lesssim \|\nabla^2 u\|_{L^2} \|\nabla^2 v\|_{L^2}$.
\end{cor}

\subsection{The equivalence between $\mathcal{H}^1$ regularity and weak sequential continuity} \label{The equivalence between H1 regularity and weak sequential continuity}
All the $\mathcal{H}^1$ regular quantities mentioned in \textsection \ref{H1 regularity of specific operators} (and numerous others that are studied in ~\cite{CLMS93}) are compensated compactness quantities. Accordingly, the following problem is presented at ~\cite[p. 267]{CLMS93}: 

\begin{quotation}
``Roughly speaking, we have on one hand `weakly continuous nonlinear quantities' and on the other hand `nonlinear quantities' that belong to $\mathscr{H}^1$. A natural -- but vague -- question is then to determine whether these two classes coincide.''
\end{quotation}

Coifman, Lions, Meyer and Semmes went on to point out that it is not clear how one should formulate precisely this very general question. Partial results are found in ~\cite{CLMS93}, ~\cite{CG92} and ~\cite{Gra92}. In this section we show that Theorem \ref{Homogeneous bilinear theorem} implies a positive result that applies to a large collection of familiar compensated fcompactness quantities. By using the potentials for divergence-free and curl-free fields constructed in \textsection \ref{Potential for divergence and curl}, the result can also be deduced from ~\cite[Theorem 1]{Gra92}. A negative result appears in \textsection \ref{H1 regular quantities not enjoying weak continuity}.

Before stating the result we note that in this context the class of multilinear quantities studied in Theorem \ref{Homogeneous bilinear theorem} is more natural than the one treated in Theorem \ref{Bilinear theorem}. Indeed, if we allow $\B$ to contain lower order derivatives, then easy examples (such as $\B(u,v) \defeq u \cdot v$ for $u,v \in W^{1,2}(\R^n,\R^m)$) show that weak sequential continuity does not imply $\mathcal{H}^1$ regularity.

\begin{cor} \label{Compensated integrability corollary}
Suppose $\M$ is as in Theorem \ref{Homogeneous bilinear theorem} and $A_j \in \{0,\dive,\curl\}$ for every $j \in \{1,\dots,r\}$. The following conditions are equivalent.
\renewcommand{\labelenumi}{(\roman{enumi})}
\begin{enumerate}
\item $\int_{\R^n} \mathcal{M}(u) = 0$ for every $u \in \Pi_{j=1}^r C_{c,A_j}^\infty(\R^n,\R^{m_j})$.

\item For every $u \in \Pi_{j=1}^r \dot{W}^{k_j,p_j}_{A_j}(\R^n,\R^{m_j})$,
\[\|\mathcal{M}(u)\|_{\mathcal{H}^1} \lesssim \prod_{j=1}^r \|\nabla^{k_j} u^j\|_{L^{p_j}}.\]

\item If the mappings $u, u^1, u^2, \cdots \in \Pi_{j=1}^r \dot{W}^{k_j,p_j}_{A_j}(\R^n,\R^{m_j})$ satisfy $u^l \rightharpoonup u$ in $\Pi_{j=1}^r \dot{W}^{k_j,p_j}(\R^n,\R^{m_j})$, then $\mathcal{M}(u^l) \to \mathcal{M}(u)$ in $\mathcal{D}'(\R^n)$.
\end{enumerate}
\end{cor}

The equivalence of (i) and (ii) follows from Proposition \ref{Density proposition} and Theorem \ref{Homogeneous bilinear theorem}, and thus it suffices to prove that (i) $\Leftrightarrow$ (iii). The result is essentially classical, but we sketch the idea of a proof for the convenience of the reader.

\begin{proof}[Proof of direction (i) $\Rightarrow$ (iii)]
Suppose (i) holds, and assume that the mappings $u, u^1, u^2, \dots$ belong to $\Pi_{j=1}^r \dot{W}^{k_j,p_j}_{A_j}(\R^n,\R^{m_j})$ and $u^l \rightharpoonup u$ in $\Pi_{j=1}^r \dot{W}^{k_j,p_j}(\R^n,\R^{m_j})$. By the estimates and equivalence of weak convergences recorded in Propositions \ref{Potential proposition}--\ref{Potential proposition 3} we may assume that $A_1 = \cdots = A_r  = 0$. When $\varphi \in \CO$, we use telescoping summation to write
\[\int_{\R^n} \varphi \M(u^l) - \int_{\R^n} \varphi \M(u)
= \sum_{j=1}^r \int_{\R^n} \varphi \M(u^{l,1},\dots,u^{l,j-1},u^{l,j}-u^j,u^{j+1},\dots,u^r),\]
and the proof is completed by using the Leibniz rule and the Rellich-Kondrachov theorem on every term.
\end{proof}

The following argument has been adapted from ~\cite[pp. 14-15]{Dac82}.

\begin{proof}[Sketch of the proof of (iii) $\Rightarrow$ (i)]
Let $u \in \prod_{j=1}^r C_{c,A_j}^\infty(\R^n,\R^{m_j})$; we aim to prove that $\int_{\R^n} \M(u) = 0$. By scaling we may assume that $\supp(u) \subset \, ]-1,1[^n$. We denote the periodic extension of $u$ from $]-1,1[^n$ to the whole space $\R^n$ by $U$. When $l \in \N$, we define $u^l \in \prod_{j=1}^r C_{c,A_j}^\infty(\R^n,\R^{m_j})$ by
\[(u^{l,1}(x),\dots,u^{l,r}(x)) \defeq	\left( \frac{1}{l^{k_1}} \chi_{]-1,1[^n}(x) U^1(lx), \dots, \frac{1}{l^{k_r}} \chi_{]-1,1[^n}(x) U^r(lx) \right).\]
Now $\nabla^{k_j} u^{l,j} \rightharpoonup 0$ for every $j \in \{1,\dots,r\}$, and so, by assumption, $\int_{\R^n} \M(u^l) \to 0$. On the other hand, a straightforward computation gives $\int_{\R^n} \M(u^l) = \int_{\R^n} \M(u)$ for every $l \in \N$, and therefore $\int_{\R^n} \M(u) = 0$.
\end{proof}

\section{The null Lagrangian condition and $\mathcal{H}^1$ regularity} \label{Homogeneous polynomials of partial derivatives}
In this section we assume that $n,m,k \in \N$. Recall from \S \ref{Null Lagrangians} that  a null Lagrangian $\LLL \colon X(n,m,k) \to \R$ is a polynomial of some degree $r \in \N$ and that $\LLL$ maps $\dot{W}^{k,r}(\R^n,\R^m)$ into $\HR$ with the natural norm bound $\|\LLL(\nabla^k u)\|_{\mathcal{H}^1} \lesssim \|\nabla^k u\|_{L^r}^r$ if and only if $\LLL$ is $r$-homogeneous. Our main aim is to show that for homogeneous polynomials, $\mathcal{H}^1$ regularity is a strictly weaker condition than weak sequential continuity. This provides, in the present context, a negative answer to the question of Coifman, Lions, Meyer and Semmes quoted in \textsection \ref{The equivalence between H1 regularity and weak sequential continuity}.

\subsection{Characterization of $\mathcal{H}^1$ regularity for polynomials}
We start by fixing notation. We denote multiples of multi-indices by capital greek letters and endow them with the following ordering.

\begin{definition}
When $\Gamma \defeq (\gamma^1,\dots,\gamma^r), \; \Theta \defeq (\theta^1,\dots,\theta^r) \in ((\N \cup \{0\})^n)^r$, we say that $\Gamma \prec \Theta$ if $\gamma^1 \le \theta^1, \dots, \gamma^r \le \theta^r$ with strict inequality for at least one index $i \in \{1,\dots,r\}$.
\end{definition}

When $r \ge 2$, a general $r$-homogeneous polynomial $\LLL \colon X(n,m,k) \to \R$ is of the form
\begin{equation} \label{Form of homogeneous polynomial}
\LLL(\nabla^k u) \defeq \sum_{\nu \in \{1,...,m\}^r} \sum_{|\theta_1|, \cdots, |\theta_r| = k} C_{\Theta,\nu} \prod_{i=1}^r \partial^{\theta^i} u_{\nu_i}.
\end{equation}

\begin{theorem} \label{Extension to homogeneous polynomials}
Let $r \ge 2$. Suppose $\LLL \colon X(n,m,k) \to \R$ is an $r$\textsuperscript{th} order polynomial. The following conditions are equivalent.

\renewcommand{\labelenumi}{(\roman{enumi})}
\begin{enumerate}
\item $\|\LLL(\nabla^k u)\|_{\mathcal{H}^1} \lesssim \|\nabla^k u\|_{L^r}^r$ for all $u \in \dot{W}^{k,r}(\R^n,\R^m)$.

\item $\int_{\R^n} \LLL(\nabla^k u) = 0$ for every $u \in C_c^\infty(\R^n,\R^m)$ and $\LLL$ is $r$-homogeneous.
\end{enumerate}

\noindent When (i) and (ii) hold, the set
\[\left\{ \sum_{j=1}^\infty \lambda^j \LLL(\nabla^k u^j) \colon \sup_{j \in \N} \|u^j\|_{W^{k,r}} \le 1, \; \sum_{j=1}^\infty |\lambda^j| < \infty \right\}\]
is of the first category in $\HR$.
\end{theorem}

When (i) holds, the $r$-homogeneity of $\LLL$ follows from a standard scaling argument. The converse direction cannot be proved by simply adapting the proof of Theorem \ref{Bilinear theorem} since $\LLL$ may contain terms where partial derivatives of $u_1$ do not appear in any of the factors. For this reason, we fix $\psi \in C_c^\infty(\R^n)$ with $\supp(\psi) \subset B(0,1)$ and $\int_{\R^n} \psi^r \neq 0$ and set $\eta \defeq \psi^r$ in the definition of the norm $\|h\|_{\mathcal{H}^1} \defeq \|\sup_{t > 0} |h * \eta_t|\|_{L^1}$.

\begin{proof}[Proof of direction (ii) $\Rightarrow$ (i)]
Let $u \in \dot{W}^{k,r}(\R^n,\R^m)$, $x \in \R^n$ and $t > 0$. Then $\psi_t(x-\cdot) u \in W^{k,r}(\R^n,\R^m)$ since $\psi_t(x-\cdot)$ is compactly supported. By (ii) and approximation,
\begin{equation} \label{Vanishing integral for homogeneous polynomials}
\int_{\R^n} \LLL \left( \nabla^k \left[ \psi \left( \frac{x-\cdot}{t} \right) u \right] \right) = 0.
\end{equation}
Next, the Leibniz rule gives
\[\begin{array}{lcl}
& & \ds \LLL \left( \nabla^k \left[ \psi \left( \frac{x-\cdot}{t} \right) u \right] \right) \\
&=& \displaystyle \sum_{\Theta,\nu} C_{\Theta, \nu} \prod_{i=1}^r \partial^{\theta^i} \left[ \psi \left( \frac{x-\cdot}{t} \right) u_{\nu_i} \right] \\
&=& \displaystyle \eta \left( \frac{x-\cdot}{t} \right) \LLL(\nabla^k u)
+ \sum_{\Theta,\nu} \sum_{\Gamma \prec \Theta} C_{\Theta,\Gamma,\nu} \prod_{i=1}^r \partial^{\theta^i-\gamma^i} \left[ \psi \left( \frac{x-\cdot}{t} \right) \right] \partial^{\gamma^i} u_{\nu_i}.
  \end{array}\]
As before, we pick a cube $Q \supset B(x,t)$ with $|Q| \lesssim_n t^n$ and use the equality above, Lemma \ref{Higher order Poincare lemma} and \eqref{Vanishing integral for homogeneous polynomials} to write
\[\begin{array}{lcl}
& & \ds \int_{\R^n} \eta_t(x-\cdot) \LLL (\nabla^k u) \\
&=& \ds \sum_{\Theta,\nu} C_{\Theta,\nu} \frac{1}{t^n} \int_Q \prod_{i=1}^r \left[ \psi \left( \frac{x-\cdot}{t} \right) \partial^{\theta^i} \left( u_{\nu_i} - P^{k-1}_Q u_{\nu_i} \right) \right] \\
&=& \ds \sum_{\Theta,\nu} \sum_{\Gamma \prec \Theta} C_{\Theta,\Gamma,\nu} \frac{1}{t^n} \int_Q \prod_{i=1}^r \partial^{\theta^i-\gamma^i} \left[ \psi \left( \frac{x-\cdot}{t} \right) \right]\partial^{\gamma^i} \left( u_{\nu_i} - P^{k-1}_Q u_{\nu_i} \right).
  \end{array}\]
Consider a single term of the sum, choose $j \in \{1,\dots,r\}$ such that $\gamma^j < \theta^j$ and let $\epsilon > 0$ be small. Then, by H\"{o}lder's inequality and Lemma \ref{Higher order Poincare lemma},
\[\begin{array}{lcl}
& &   \ds \left| \frac{1}{t^n} \int_Q \prod_{i=1}^r \partial^{\theta^i-\gamma^i} \left[ \psi \left( \frac{x-\cdot}{t} \right) \right]\partial^{\gamma^i} \left( u_{\nu_i} - P^{k-1}_Q u_{\nu_i} \right) \right| \\
&\lesssim& \ds t^{-|\theta^j-\gamma^j|} \left( \Aint_Q \left| \partial^{\gamma^j} \left( u_{\nu_j} - P^{k-1}_Q u_{\nu_j} \right) \right|^\frac{r-\epsilon}{1-\epsilon} \right)^\frac{1-\epsilon}{r-\epsilon} \\
&\cdot& \ds \prod_{i \neq j} t^{-|\theta^i-\gamma^i|} \left( \Aint_Q \left| \partial^{\gamma^i} \left( u_{\nu_i} - P^{k-1}_Q u_{\nu_i} \right) \right|^{r-\epsilon} \right)^\frac{1}{r-\epsilon} \\
&\lesssim& \ds [M_{r-\epsilon} (\nabla^k u)(x)]^r.
  \end{array}\]
We conclude via Hardy-Littlewood Maximal Theorem as before.

\end{proof}

\begin{proof}[Completion of the proof of Theorem \ref{Extension to homogeneous polynomials}]
Seeking contradiction, assume
\[\left\{ \sum_{j=1}^\infty \lambda^j \LLL(\nabla^k u^j) \colon \sup_{j \in \N} \|u^j\|_{W^{k,r}} \le 1, \; \sum_{j=1}^\infty |\lambda^j| < \infty \right\}\]
is of the second category in $\HR$. Set $V \defeq \{\pm \LLL(\nabla^k u) \colon \|u\|_{W^{k,r}} \le 1\} \subset \HR$; by condition (i) of Theorem \ref{Extension to homogeneous polynomials}, $V$ is bounded.
Now condition (iv) of Lemma \ref{Main lemma} is satisfied, and therefore
\begin{equation} \label{False inequality for P 2}
\|b\|_{\text{BMO}} \lesssim \sup_{h \in V} \int_{\R^n} bh = \sup_{u \neq 0} \frac{\left| \int b \LLL(\nabla^k u) \right|}{\|u\|_{W^{k,r}}^r}
\end{equation}
for all $b \in \BMO$. We will show that \eqref{False inequality for P 2} leads to a contradiction.

\vspace{0.2cm}
When $b \in \CO \setminus \{0\}$, we use the Leibniz rule as in the proof above to get
\begin{equation} \label{Integration by parts}
\int_{\R^n} b^r \LLL(\nabla^k u) = \sum_{\Theta,\nu} \sum_{\Gamma \prec \Theta} C_{\Theta,\Gamma,\nu} \int_{\R^n} \prod_{i=1}^r \partial^{\theta^i-\gamma^i} b \, \partial^{\gamma^i} u_{\nu_i}.
\end{equation}
As in the proof of Theorem \ref{Main theorem}, we fix $b \in \CO \setminus \{0\}$ and a parameter $\tau > 0$ and use \eqref{False inequality for P 2} to select $\tilde{u} = u(\tau \cdot) \in W^{k,r}(\R^n,\R^m) \setminus \{0\}$ satisfying
\begin{equation} \label{False norm estimate 2}
\|b^r\|_{\text{BMO}} = \|b^r(\tau \cdot)\|_{\text{BMO}} \lesssim \frac{\displaystyle \left| \int_{\R^n} b^r(\tau \cdot) \LLL(\nabla^k [u(\tau \cdot]) \right|}{\|u(\tau \cdot)\|_{W^{k,r}}^r}.
\end{equation}

By \eqref{Integration by parts} and a change of variables,
\[\begin{array}{lcl}
& & \displaystyle  \int_{\R^n} b^r(\tau \cdot) \LLL(\nabla^k [u(\tau \cdot]) \\
&=& \displaystyle \sum_{\Theta,\sigma} \sum_{\Gamma \prec \Theta} C_{\Theta,\Gamma,\nu} \int_{\R^n} \prod_{i=1}^r \partial^{\theta^i-\gamma^i} [b(\tau \cdot)] \, \partial^{\gamma^i} [u_{\nu_i}(\tau \cdot)] \\
&=& \displaystyle \sum_{\Theta,\sigma} \sum_{\Gamma \prec \Theta} C_{\Theta,\Gamma,\nu} \tau^{-n} \int_{\R^n} \prod_{i=1}^r \tau^{|\theta^i|} \partial^{\theta^i-\gamma^i} b \, \partial^{\gamma^i} u_{\nu_i} \\
&\hspace{0.15cm} \lesssim_b& \displaystyle \sum_{\Theta,\sigma} \sum_{\Gamma \prec \Theta} \tau^{-n} \int_{\R^n} \left| \prod_{i=1}^r \tau^{|\theta^i|} \partial^{\gamma^i} u_{\nu_i} \right|.
  \end{array}\]
We next control a single term by the generalized Young's inequality. Choosing $j \in \{1,\dots,r\}$ such that $\gamma_j < \theta_j$ we get
\[\begin{array}{lcl}
      \displaystyle \int_{\R^n} \left| \prod_{i=1}^r \tau^{|\theta^i|} \partial^{\gamma^i} u_{\nu_i} \right|
&=& \displaystyle \int_{\R^n} \left| M \tau^{|\theta^j|} \partial^{\gamma^j} u_{\nu_j} \prod_{i \neq j} \frac{\tau^{|\theta^i|}}{M^\frac{1}{r-1}} \partial^{\gamma^i} u_{\nu_i} \right| \\
&\le& \displaystyle \frac{\tau^{r |\theta^j|} M^r}{r} \int_{\R^n} |\partial^{\gamma^j} u_{\nu_j}|^r + \sum_{i \neq j} \frac{\tau^{r |\theta^i|}}{r M^{r'}} \int_{\R^n} |\partial^{\gamma^i} u_{\nu_i}|^r.
  \end{array}\]
By combining the inequalities and using \eqref{False norm estimate 2} we conclude that
\[\begin{array}{lcl}
&& \displaystyle \sum_{\Theta,\nu} \sum_{\Gamma \prec \Theta, \gamma^j < \theta^j} \left( \tau^{r |\gamma^j| - n} \int_{\R^n} |\partial^{\gamma^j} u_{\nu_j}|^r + \sum_{i \neq j} \tau^{r |\gamma^i| - n} \int_{\R^n} |\partial^{\gamma^i} u_{\nu^i}|^r \right) \\
&=& \displaystyle \sum_{\Theta,\nu} \sum_{\Gamma \prec \Theta, \gamma^j < \theta^j} \left( \int_{\R^n} |\partial^{\gamma^j} [u_{\nu_j}(\tau \cdot)|^r + \sum_{i \neq j} \int_{\R^n} |\partial^{\gamma^i} [u_{\nu_i}(\tau \cdot)]|^r \right) \\
&\lesssim& \displaystyle \|u(\tau \cdot)\|_{W^{k,r}}^r \\
&\hspace{0.11cm} \lesssim_b& \displaystyle \sum_{\Theta,\nu} \sum_{\Gamma \prec \Theta, \gamma^j < \theta^j} \left( \tau^{r |\theta^j|-n} M^r \int_{\R^n} |\partial^{\gamma^j} u_{\nu_j}|^r + \sum_{i \neq j} \frac{\tau^{r |\theta^i|-n}}{M^{r'}} \int_{\R^n} |\partial^{\gamma^i} u_{\nu_i}|^r \right).
  \end{array}\]
For every term of the series we have $|\gamma^j| < |\theta^j|$ and $|\gamma^i| \le |\theta^i|$ for all $i \neq j$, and once again we get a contradiction by choosing first $M$ large enough and then $\tau$ small enough.
\end{proof}

\subsection{$\mathcal{H}^1$ regular quantities not enjoying weak continuity} \label{H1 regular quantities not enjoying weak continuity}
Theorem \ref{Extension to homogeneous polynomials} allows us to find, for homogeneous polynomials of partial derivatives, a negative answer to the question of Coifman, Lions, Meyer, and Semmes quoted at \textsection \ref{The equivalence between H1 regularity and weak sequential continuity}. There exist homogeneous polynomials $\LLL$ satisfying $\int_{\R^n} \LLL(\nabla^k \varphi) = 0$ for all $\varphi \in C_c^\infty(\R^n,\R^m)$ but not being null Lagrangians. While this fact is undoubtedly well-known to experts, we give an explicit example for the reader's convenience.

\begin{prop} \label{Counterexample proposition}
Define $\LLL \colon X(3,2,2) \to \R$ via
\[\begin{array}{lcl}
    \ds \LLL(\nabla^2 (u,v))
&\defeq& \ds \partial_{xx} u \partial_{yy} v \partial_{zz} v - \frac{1}{2} \partial_{xx} u \partial_{yz} v \partial_{yz} v - \partial_{xy} u \partial_{xy} v \partial_{zz} v \\
&+& \ds \partial_{yz} u \partial_{xx} v \partial_{yz} v - \partial_{zz} u \partial_{xx} v \partial_{yy} v + \frac{1}{2} \partial_{zz} u \partial_{xy} v \partial_{xy} v.
  \end{array}\]
Then $\|\LLL(\nabla^2 (u,v))\|_{\mathcal{H}^1} \lesssim \|\nabla^2(u,v)\|_{L^3}^3$ for every $(u,v) \in \dot{W}^{2,3}(\R^3,\R^2)$ and yet $(u^l,v^l) \rightharpoonup (u,v)$ in $\dot{W}^{2,3}(\R^3,\R^2)$ does not imply that $\LLL(\nabla^2(u^l,v^l)) \to \LLL(\nabla^2(u,v))$ in $\mathcal{D}'(\R^3)$.
\end{prop}

\begin{proof}
When $u,v \in C_c^\infty(\R^3)$, an integration by parts gives
\[\begin{array}{lcl}
    \ds \int_{\R^3} \LLL(\nabla^2 (u,v))
&=& \ds \sum_{|\alpha|,|\beta|,|\gamma|=2} c_{\alpha,\beta,\gamma} \int_{\R^n} \partial^\alpha u \partial^\beta v \partial^\gamma v \\
&=& \ds \sum_{|\beta'|+|\gamma'|=6} d_{\beta',\gamma'} \int_{\R^n} u \partial^{\beta'} v \partial^{\gamma'} v = 0
  \end{array}\]
since, as is easily checked, $d_{\beta',\gamma'} = 0$ for every $\beta'$ and $\gamma'$. By Theorem \ref{Extension to homogeneous polynomials}, $\|\LLL(\nabla^2 (u,v))\|_{\mathcal{H}^1} \lesssim \|\nabla^2(u,v)\|_{L^3}^3$ for all $(u,v) \in \dot{W}^{2,3}(\R^3,\R^2)$.

However, $\LLL$ is not weakly sequentially continuous. Indeed, consider $u = u^l \defeq x^2 \psi_B$, where $\psi$ is a smooth cutoff function for a ball $B \subset \R^3$, and $v^l(x,y,z) \defeq l^{-2} \sin(ly+lz) \psi_B$. Then $(u^l,v^l) \rightharpoonup (u,0)$ in $\dot{W}^{2,3}(\R^3,\R^2)$ but $\LLL(\nabla^2(u^l,v^l)) = \sin^2(ly+lz)$ in $B$ and so $\LLL(\nabla^2(u^l,v^l))$ does not converge to $0$ in $\mathcal{D}'(\R^3)$.
\end{proof}

The basic idea of Proposition \ref{Counterexample proposition} is that the integration by parts creates cancellations of partial derivatives of $v$ but such cancellations are not present in integrals $\int_{\R^3} \varphi \LLL(\nabla^2 (u,v))$, where $\varphi \in C_c^\infty(\R^3)$.

\section{Open problems related to Question \ref{CLMS question}} \label{Open questions}
In this section we discuss open questions on the surjectivity of the Jacobian operator in Hardy spaces. The core of all the non-surjectivity results of this article is the incompatibility of the scaling properties of inhomogeneous Sobolev spaces and the target space $\HR$, and this particular obstacle to surjectivity is not present when homogeneous Sobolev spaces are used as domains of definition of the operators or the surjectivity question is treated in bounded domains.

\subsection{Homogeneous Sobolev spaces} \label{Related problems and conjectures}
Coifman, Lions, Meyer, and Semmes deduced Theorem \ref{Jacobian decomposition theorem} by proving the norm equivalence
\begin{equation} \label{CLMS norm equivalence}
\|b\|_{\text{BMO}} \approx \sup \left\{ \int_{\R^n} b J u \colon u \in \W, \; \int_{\R^n} |\D u|^n \le 1 \right\}
\end{equation}
and using direction $(i) \Rightarrow (iii)$ of Lemma \ref{Main lemma}. In view of \eqref{CLMS norm equivalence} it is natural to ask whether $J$ maps the inhomogeneous Sobolev space $\dot{W}^{1,n}(\R^n,\R^n)$ onto $\HR$. Recall that when $1 < p < \infty$, the Hardy space $\mathcal{H}^p(\R^n)$ coincides with $L^p(\R^n)$. The following conjecture was posed by T. Iwaniec in 1997.

\begin{conjecture}[~\cite{Iwa97}] \label{Tadeusz conjecture}
For every $n \ge 2$ and every $p \in [
1,\infty[$ the Jacobian operator $J \colon \dot{W}^{1,np}(\R^n,\R^n) \to \mathcal{H}^p(\R^n)$ has a continuous right inverse.
\end{conjecture}

In ~\cite{Lin15} the author made progress on Conjecture \ref{Tadeusz conjecture} in the case $n=2, p=1$ via calculus of variations and Banach space geometry. Recall the norm $\|\cdot\|_{\text{BMO}_\SSS}$ defined in \eqref{Definition of BMOS norm} and denote the dual norm of ${\mathcal H}^1(\C)$ by $\|\cdot\|_{{\mathcal H}^1_\SSS}$. By the definition of $\|\cdot\|_{\text{BMO}_\SSS}$,
\[\|J u\|_{{\mathcal H}^1_\SSS} = \sup_{\substack{b \in \operatorname{VMO}(\C) \\ \|b\|_{\text{BMO}_\SSS}=1}} \frac{\int_\C b J u}{\int_\C |u_{\bar{z}}|^2} \int_\C |u_{\bar{z}}|^2 \le \int_\C |u_{\bar{z}}|^2\]
for all (non-constant) $u \in \dot{W}^{1,2}(\C,\C)$ and the inequality is sharp.

There exists, however, a large set of datas $h \in \HR$ for which the equation $J u = h$ has a solution satisfying $\int_\C |u_{\bar{z}}|^2 = \|J u\|_{{\mathcal H}^1_\SSS}$ -- in particular, the set is closed and contains all the extreme points of the unit ball of $(\HR, \|\cdot\|_{\mathcal{H}^1_\SSS})$ (see ~\cite[Theorems 1.19 $\&$ 1.27]{Lin15}). We refer to ~\cite{Lin15} and the forthcoming article ~\cite{Lin16} for more information on the topic. The author feels that there is already enough evidence for the following conjecture which appears as a question in ~\cite{Lin15}.

\begin{conjecture} \label{Oma konjektuuri}
For every $h \in {\mathcal H}^1(\C)$ there exists $u \in \dot{W}^{1,2}(\C,\C)$ such that $J u = h$ and $\|J u\|_{{\mathcal H}^1_S} = \int_\C |u_{\bar{z}}|^2$.
\end{conjecture}

A positive answer to Conjecture \ref{Tadeusz conjecture} would imply, among other things, a new characterization of the duality pairing of $b \in \operatorname{BMO}(\C)$ and $h \in \mathcal{H}^1(\C)$ as an $L^2$ inner product via formula \eqref{AIM formula}:
\[\int_\C bh = \int_\C f \overline{K_b f},\]
where $|\SSS f|^2 - |f|^2 = J_u = h$. Conjecture \ref{Oma konjektuuri} in turn would imply, for instance, that we would have precise control on the norms of all the mappings involved: $\int_\C |f|^2 = \|h\|_{{\mathcal H}^1_\SSS}$ and $\|K_b\|_{L^2 \to L^2} = \|b\|_{\text{BMO}_\SSS}$.

\subsection{Bounded strongly Lipschitz domains}
When a domain $\Omega \subset \R^n$ is bounded and strongly Lipschitz, the estimate $\|Ju\|_{\mathcal{H}^1_z(\Omega)} \lesssim_\Omega \prod_{j=1}^n \|\nabla u_j\|_{L^n}$ holds for all $u \in W^{1,n}_0(\Omega,\R^n)$. In ~\cite{LYS05} Z.J. Lou, S.Z. Yang, and D.J. Song asked, in analogy to Question \ref{CLMS question}, whether $J$ maps $W_0^{1,2}(\Omega, \R^2)$ onto $\Hz$.

The author gave a negative answer for all (nonempty) bounded Lipschitz domains  of $\R^2$ in ~\cite[Theorem 8.12]{Lin15} by using the theory of T. Iwaniec and V. Sver\'{a}k on mappings of integrable distortion (see ~\cite{IS93}). We pose the following conjecture which appears as a question in ~\cite{Lin15}.

\begin{conjecture} \label{My question}
The Jacobian maps $W^{1,2}(\Omega) \times W_0^{1,2}(\Omega)$ onto $\Hz$.
\end{conjecture}

In view of the theory of the Beurling transform on domains, Conjecture \ref{My question} is more natural than the one for $J \colon W_0^{1,2}(\Omega,\R^2) \to \Hz$. In higher dimensions the correct analogue of Conjecture \ref{My question} appears to be whether $J$ maps $W^{1,n}(\Omega,\R^{n-1}) \times W_0^{1,n}(\Omega)$ onto $\mathcal{H}^1_z(\Omega)$.

\subsection{$L^p$ spaces} \label{The Jacobian equation in Lp spaces}
We start the discussion on the Jacobian equation with data in $L^p(\R^n)$ by presenting the following analogue of Theorem \ref{Main theorem}.

\begin{theorem} \label{Main theorem 3}
For every $n \ge 2$ and $p \in ]1,\infty[$, the set
\[\left\{ \sum_{j=1}^\infty \lambda^j J u^j \colon \|u^j\|_{W^{1,np}} \le 1 \text{ for every } j \in \N, \, \; \sum_{j=1}^\infty |\lambda^j| < \infty \right\}\]
is of the first category in $L^p(\R^n)$.
\end{theorem}

Theorem \ref{Main theorem 3} is demonstrated by a proof similar to that of Theorem \ref{Main theorem} by using the fact that $\|b(\tau \cdot)\|_{L^{p'}} = \tau^{-n/p'} \|b\|_{L^{p'}}$ for all $b \in L^{p'}(\R^n)$ and all $\tau > 0$. The only notable difference is the use of the integration by parts -- in this instance we make a further application of H\"{o}lder's inequality and estimate
\[\begin{array}{lcl}
      \displaystyle \int_{\R^n} b J u
&\le& \displaystyle \int_{\supp(b)} \left( M^n |u|^n + \frac{C(b)}{M^{n'}} |\D u|^n \right) \\
&\le& \displaystyle C(b,p) \left( M^{np} \int_{\R^n} |u|^{np} + \frac{1}{M^{n'p}} \int_{\R^n} |\D u|^{np} \right)^\frac{1}{p}.
  \end{array}\]
The rest of the details are left to the reader.

An important step towards solving Conjecture \ref{Tadeusz conjecture} for $p \in ]1,\infty[$ would be to find out whether the following analogue of \eqref{CLMS norm equivalence} holds:
\begin{equation} \label{Hypothetical norm equivalence}
\|b\|_{L^{p'}} \approx \sup_{\|\D u\|_{L^{np}} \le 1} \int_{\R^n} b J u \qquad \text{for all } b \in L^{p'}(\R^n).
\end{equation}
Condition \eqref{Hypothetical norm equivalence} is necessary for surjectivity of $J \colon \dot{W}^{1,np}(\R^n,\R^n)\to L^p(\R^n)$: setting $V \defeq \{J u \colon \|\D u\|_{L^{np}} \le 1\} \subset L^p(\R^n)$ we deduce from Lemma \ref{Main lemma} that $\left\{ \sum_{j=1}^\infty \lambda^j J u^j \colon \|\D u^j\|_{L^{np}} \le 1 \text{ for every } j \in \N, \, \; \sum_{j=1}^\infty |\lambda^j| < \infty \right\}$ is of the second category in $L^p(\R^n)$ if and only if \eqref{Hypothetical norm equivalence} holds. Conversely, if \eqref{Hypothetical norm equivalence} holds, then many techniques of ~\cite{Lin15} formulated for $\mathcal{H}^1(\C)$ are available in $L^p(\R^n)$, especially in the planar case.

When $n = 2$, condition \eqref{Hypothetical norm equivalence} has an interesting connection to operator norms of commutators $\mathcal{S} b - b \mathcal{S}$. Indeed, using formula \eqref{AIM formula} we get
\[
\sup_{\|\D u\|_{L^{np}} \le 1} \int_{\R^n} b J u \approx \sup_{\|f\|_{L^{2p}} \le 1} \int_\C f \overline{K_b f} \approx \|\mathcal{S} b - b \mathcal{S}\|_{L^{2p} \to L^{(2p)'}}\]
for every $b \in L^{p'}(\C)$, where the last two-sided estimate is proved by adapting the proof of equality of the operator norm and numerical radius of a self-adjoint operator in a Hilbert space. The author has been unable to prove or disprove \eqref{Hypothetical norm equivalence}.

\section*{Acknowledgement}
The author wishes to thank Kari Astala and Tadeusz Iwaniec for useful discussions on the subject matter of the article.

\bibliography{mybibliography}
\bibliographystyle{amsplain}

\end{document}